\documentclass[a4paper,11pt]{article}
\usepackage{graphics,color}
\usepackage{amsmath}
\usepackage{amssymb}
\usepackage{mathrsfs}
\usepackage{cite}
\usepackage{verbatim}
\usepackage{float}
\usepackage{graphicx}
\usepackage{amsthm}
\usepackage{textcomp}
\usepackage{subfig}
\usepackage{hyperref}

\newcommand{\area}{\mathcal A}

\newcommand{\axialcoordofcylinder}{{w_1}}

\newcommand{\BallR}{{\Omega}}
\newcommand{\Balleps}{{\sourcedisk_{r_k}}}

\newcommand{\Coneeps}{{{C}_k}}

\newcommand{\deltaeps}{\delta_ k}
\newcommand{\dirdatum}{\varphi}

\newcommand{\domalaa}{X_{2\longR}^{\rm conv}}
\newcommand{\doubledrectangle} {R_{2\longR}}

\newcommand{\FB}{{\mathcal F}_{2\longR}}
\newcommand{\FBl}{{\mathcal F}_\longR}

\newcommand\grad{\nabla}

\newcommand{\h}{h}

\newcommand{\Hone}{\mathcal{H}^1}
\newcommand{\Htwo}{\mathcal{H}^2}

\newcommand{\Hspace}{\mathcal H_{2l}}
\newcommand{\hstar}{{\h^{\star}}}
\newcommand{\hstark}{{\h^{\star}_k}}

\newcommand{\invr}{H_k}
\newcommand{\invtheta}{\Theta_k}
\newcommand{\jump}[1]{\text{{\rm \textlbrackdbl}}{#1}\text{{\rm \textrbrackdbl}}}

\newcommand{\Lip}{{\rm Lip}}

\newcommand{\longR}{l}

\newcommand{\M}{\mathcal{M}}

\newcommand{\nada}[1]{}
\def\NN{\mathbb{N}}

\newcommand{\OmAlaa}{\doubledrectangle}
\newcommand{\Omegah}{\subgraph_\h}

\newcommand{\Omegahstar}{\subgraph_{\hstar}}
\newcommand{\Omegahstark}{\subgraph_{\hstark}}

\newcommand{\Om}{\Omega}

\newcommand{\psione}{\psi}

\newcommand{\psionestar}{\psi^\star}
\newcommand{\psionekstar}{\psi^\star_k}
\newcommand{\partialbar}{ \partial_D}

\newcommand{\R}{\mathbb{R}}

\newcommand{\relarea}{\overline \area}
\newcommand{\reps}{{r_k}}
\newcommand{\res}{\mathop{\hbox{\vrule height 7pt width 0.5pt depth 0pt
\vrule height 0.5pt width 6pt depth 0pt}}\nolimits}

\newcommand{\scoord}{{w_2}}
\newcommand{\scalarfunction}{\psi}

\newcommand{\seps}{s_k}

\newcommand{\Sone}{\mathbb{S}^1}
\newcommand{\sourcedisk}{{\rm B}}
\newcommand{\sourceradialcoordinate}{r}
\newcommand{\sourceangularcoordinate}{\alpha}

\newcommand{\subgraph}{SG}
\newcommand{\subgraphh}{SG_h}

\newcommand{\tcoord}{{w_1}}
\newcommand{\thetaeps}{\theta_ k}
\newcommand{\thetaepsbar}{\overline \theta_ k}
\newcommand{\teps}{\tau_k}
\newcommand{\Teps}{T_ k}

\newcommand{\vmap}{u}
\newcommand{\veps}{{\vmap_ k}}

\newcommand{\vortexmap}{u}

\numberwithin{equation}{section}
\mathchardef\emptyset="001F

\newtheorem{theorem}{Theorem}[section]
\newtheorem{definition}[theorem]{Definition}

\newtheorem{lemma}[theorem]{Lemma}

\theoremstyle{definition}
\newtheorem{remark}[theorem]{Remark}
\newtheorem{Remark}[theorem]{Remark}

\title{The $L^1$-relaxed area of the graph of the vortex map: optimal
	upper bound}
\author{
Giovanni Bellettini\footnote{
Dipartimento di Ingegneria dell'Informazione e Scienze Matematiche, Universit\`a di Siena, 53100 Siena, Italy,
and International Centre for Theoretical Physics ICTP,
Mathematics Section, 34151 Trieste, Italy.
E-mail: giovanni.bellettini@unisi.it
                      }\and
Alaa Elshorbagy\footnote{
Technische Universit\"at Dortmund, 
Fakult\"at f\"ur Mathematik, 
44227 Dortmund, Germany. E-mail: elshorbagy.alaa1@gmail.com 
                         }
                          \and
Riccardo Scala\footnote{ 
Dipartimento di Ingegneria dell'Informazione e Scienze Matematiche, Universit\`a di Siena, 53100 Siena, Italy.
E-mail: riccardo.scala@unisi.it
}
}

\begin{document}

\maketitle

\begin{abstract}
We compute an upper bound for the value of the $L^1$-relaxed area of the graph of the 
vortex map $\vortexmap : \sourcedisk_\longR(0)\subset \R^2 
\to \R^2$, $\vortexmap(x):= x/\vert x\vert$,
$x \neq 0$, for all values of $\longR>0$. 
Together with a previously proven lower bound, this upper bound turns out
to be optimal.
Interestingly, for the radius $\longR$ in 
a certain range, in particular $\longR$ not too large, a Plateau-type 
problem, having as solution 
a sort of catenoid constrained to contain a segment,  has to be solved.
\end{abstract}


\vspace{2mm}

\noindent {\bf AMS (MOS) subject clas\-si\-fi\-ca\-tion:} 
49Q15, 49Q20, 49J45, 58E12.

%%%%%%%%%%%%%%%%%%%%%%%%%%%%%%%%%%%%%%%%%%%%%%%%%%%%%%%%%%%%%%%%%%%%%%%%%
\section{Introduction}\label{sec:introduction}
%%%%%%%%%%%%%%%%%%%%%%%%%%%%%%%%%%%%%%%%%%%%%%%%%%%%%%%%%%%%%%%%%%%%%%%%%
Determining the domain and the expression of 
the relaxed area functional of graphs of nonsmooth maps 
in codimension greater than $1$ 
is a challenging problem whose solution is far from being reached. 
Let $\Omega \subset \R^n$  be a bounded open set and let $v:\Omega \rightarrow\R^N$ be a map of class $C^1$; 
the graph area  of $v$ over $\Omega$ is given by 
\begin{align}\label{area_classical}
 \area(v,\Omega)=\int_\Omega|\mathcal M(\nabla v)|~dx,
\end{align}
where $\mathcal M(\nabla v)$ is the vector whose entries are the 
determinants of the minors of the gradient $\nabla v$ of $v$ 
of all orders\footnote{By convention, the determinant
of order $0$ is $1$.} $k$, $0\leq k\leq\min\{n,N\}$.
In order to extend this  functional out of $C^1(\Omega,\R^N)$, one is led to
define, for any $v \in L^1(\Omega,\R^N)$, 
\begin{equation}\label{area_relax}
 \relarea
(v,\Omega):=\inf
\Big\{\liminf_{k\rightarrow +\infty}\area(v_k,\Omega)\Big\},
\end{equation}
which is  called the \emph{(sequential) relaxed area functional}. 
The infimum appearing in \ref{area_relax} is computed 
among all possible sequences of maps $v_k\in C^1(\Omega,\R^N)$ 
tending to  $v$ in $L^1(\Omega,\R^N)$. 
The results of Acerbi and Dal Maso \cite{AcDa:94} show that $\relarea(\cdot,\Om)$ 
extends $\area(\cdot,\Om)$ and is $L^1$-lower semicontinuous. This procedure of relaxation, 
besides extending the notion of graph's area to non-smooth maps, is 
needed also because $\mathcal A(\cdot,\Omega)$ 
is not   $L^1$-lower semicontinuous\footnote{When $n=N=2$, there 
are sequences $(v_k)\subset W^{1,p}(\Om, \R^2)$,
with $p \in [1,2)$,
weakly converging in $W^{1,p}(\Omega,\R^2)$ 
to a smooth map $v$ for which $\mathcal A(v,\Omega)>
\limsup_{k\rightarrow+\infty}\mathcal A(v_k,\Omega)$,
where $\area(v_k,\Omega)$ is defined in the same form
as for $C^1$-maps in 
\eqref{area_classical}, with
the determinant of $\grad v_k$ intended in the almost everywhere pointwise sense;
see \cite[Counterexample 7.4]{Ball_Murat:84} and \cite{AcDa:94}.
This counterexample must be slightly modified, considering
$u_k(x) = k x + \lambda (x/\Vert x\Vert_\infty -x)$
for $x \in [-1/k,1/k]$, with $\lambda > 0$ satisfying $(1+\lambda^2)/2 > \sqrt{1+\lambda^2}$,
in order to get the strict inequality above.
}, in 
contrast with  similar polyconvex functionals that enjoy a growth condition of the form $F(u)\geq C|\mathcal M(\nabla u)|^p$ for some $C>0$, and suitable $p>1$ (see, {\it e.g.},  \cite{Morrey:66,Dacorogna:89,FuHu:95}).

When $N=1$ it is possible to characterize 
the domain of $\relarea(\cdot,\Omega)$ 
and its expression \cite{DalMaso:80}:
$\relarea(v,\Omega)$ 
is finite 
if and only if $v\in BV(\Omega)$, in which case 
\begin{align}\label{integral_formula}
\relarea(v,\Omega)= 
\int_\Omega\sqrt{1+|\nabla v|^2}dx+|D^sv|(\Omega),
\end{align}
$\nabla v$ and $D^sv$ representing the absolutely continuous and singular parts of the distributional gradient $Dv$ of $v$. 
Formula \eqref{integral_formula} gives a classical 
example of non-parametric variational integral. This turns out to be a measure 
when considered as a function of $\Omega$
 (and the map $u$ being fixed \cite{GS:64}), and has several applications, as for instance 
in capillarity
problems \cite{Finn:86} and  in the analysis 
of the Cartesian Plateau problem \cite{Giusti:84}. 
The higher codimension case, namely $N>1$,
is much more involved and, once again, has as  main motivation 
the study of the Cartesian Plateau problem in higher codimension;  from a theoretical point of view, it is of indipendent interest in  
 Calculus of Variations questions  involving nonconvex integrands with
nonstandard growth (see, e.g., \cite{Ball:77,Dacorogna:89,GiMoSu:98}).

In this paper we restrict  our attention to the first non-standard case, namely $n=N=2$.
For a map $v \in C^1(\Omega, \R^2)$ and $\Omega \subset \subset \R^2$, the quantity
$\mathcal A(v,\Omega)$
coincides with the area of the graph 
$G_v:=\{(x,y)\in \Omega\times\R^2:y=v(x)\}$ of $v$ seen as a 
Cartesian surface of codimension $2$ in $\Omega\times\R^2$, and is given by
\begin{equation*}
 \area(v,\Omega)
=\int_\Omega\sqrt{1+|\nabla v(x_1,x_2)|^2+|Jv(x_1,x_2)|^2}~dx_1dx_2.
\end{equation*}
Here $\nabla v$ is the gradient of $v$, a $2\times2$ matrix, 
$\vert \nabla v\vert^2$ is the sum of the squares of all
elements of $\nabla v$, 
and $Jv$ is the Jacobian determinant of $v$, {\it i.e.}, the determinant of $\nabla v$.
It is worth to point out once more a couple of relevant difficulties
arising when the codimension is greater than $1$: the 
functional $\mathcal A(\cdot,\Omega)$ is 
no longer convex, but just polyconvex;  in addition
it has a sort of unilateral linear growth, in 
the sense that it is bounded below, but not necessarily above, by the total
variation of $v$. 
A characterization of the domain of $\relarea(\cdot, \Omega)$ and of its expression is,
at the moment, not available. Specifically, it is 
only known that the domain of $\relarea(\cdot,\Omega)$ 
is a proper subset of $BV(\Omega,\R^2)$, 
and that integral representation 
formulas such as \eqref{integral_formula} (on the domain of 
$\relarea(\cdot, \Omega)$) are not
possible. This is due to the 
additional difficulty that in general, for a fixed map $v$, 
the set function $A\subseteq \Omega \mapsto \relarea (v,A)$ 
may be not subadditive and,in such a case, 
it cannot be a measure (as opposite
to what happens in codimension $1$ for a large class of non-parametric variational integrals \cite{GS:64}). 
This interesting phenomenon 
was conjectured by De Giorgi \cite{DeGiorgi:92} for the 
triple junction map $u_T:\Omega=\sourcedisk_\longR(0)\rightarrow\R^2$, and proved in \cite{AcDa:94}, where the authors exhibited 
three subsets $\Omega_1, \Omega_2, \Omega_3$ of the open disk
$\sourcedisk_\longR(0)$ of radius $\longR$ centered at $0$,
such that 
\begin{align}\label{non-locality}
 \Omega_1\subset \Omega_2\cup\Omega_3\qquad\text{and }\qquad \relarea 
(u_T,\Omega_1)>\relarea(u_T,\Omega_2)+\relarea(u_T,\Omega_3).
\end{align}
The triple junction map $u_T \in BV(\Omega,\R^2)$ 
takes only three values $\alpha,\beta,\gamma\in \R^2$, 
the vertices of an equilateral triangle, in three circular $120^o$-degree
sectors of $\Omega$ meeting at $0$.
The same authors show that  the non-locality property \eqref{non-locality} holds also for the Sobolev
map $u(x)=\frac{x}{|x|}$, called here the vortex map,  where $\Omega$ is the open 
ball $B_\longR(0)$ of radius $\longR$ centered at the origin,
the singular point, and $n=N \geq 3$.
For these two maps $u_T$ and $u$ much effort
has been done to understand the exact value of the area 
functional;
the corresponding geometric problem stands in finding the optimal
way, in terms of area, to ``fill the holes'' of the graph 
of $u_T$ and $u$ (two non-smooth $2$-dimensional sets of 
codimension two) with limits
of sequences of smooth two-dimensional graphs. 
 In \cite{AcDa:94} it is proved that both $u_T$ and $u$ have finite relaxed area, but 
only lower and upper bounds were available for $u_T$, whereas the sharp estimate for $u$ is provided only for $l$ large enough. For the 
triple junction map
 $u_T$ an improvement is obtained in 
\cite{BePa:10}, where it is exhibited a sequence $(u_k)$ of Lipschitz maps 
$u_k:\sourcedisk_\longR(0)\rightarrow\R^2$ 
converging to $u$ in $L^1(\Om, \R^2)$, such 
that
\begin{align*}
\lim_{k \to +\infty}
 \area(u_k,\sourcedisk_\longR(0)) = |\mathcal G_{u_T}|+3m_\longR,
\end{align*}
where $|\mathcal G_{u_T}|$ is the area of the graph of $u_T$ 
out of the jump set, and $m_\longR$ 
is the area of an area-minimizing surface, solution of a Plateau-type problem in $\R^3$. 
Roughly speaking, three entangled area-minimizing surfaces with area $m_\longR$ 
(each sitting in a copy of $\R^3 \subset \R^4$, the three $\R^3$'s being 
mutually nonparallel)
are needed in $\sourcedisk_\longR(0)
\times \R^2$ to ``fill the holes'' left by the graph $\mathcal G_{u_T} $ of $u_T$, which is not boundaryless ({\it i.e.}, 
the boundary as a 
current is nonzero). The optimality of $(u_k)$ was also conjectured
in \cite{BePa:10}, and proven subsequently in  
\cite{Scala:19}, where a crucial tool is 
a symmetrization technique for boudaryless 
integral currents.

In the present paper we instead focus on the  \emph{vortex map} $u$ in $n=2$ dimensions, 
and provide the optimal upper bound for $\relarea(u,\sourcedisk_\longR(0))$, for all $l>0$. The vortex map, that is  
\begin{equation}\label{vortexmapdef}
 \vortexmap(x):=\frac{x}{|x|}, \qquad x \in \Om \setminus \{0\}, \ 
\Om = \sourcedisk_\longR(0) \subset \R^2, 
\end{equation}
 belongs to $W^{1,p}(\Omega,\R^2)$ for all $p \in [1,2)$, but not to $W^{1,2}(\Omega, \R^2)$),
and its image  is the one-dimensional unit circle $\mathbb S^1\subset\R^2$, 
so that  $Ju(x)= {\rm det}(\grad u(x)) =0$ for all 
$x\in \Omega \setminus \{0\}$. 
In \cite[Lemma 5.2]{AcDa:94}, the authors show\footnote{In \cite{AcDa:94} the 
proof of \eqref{valueA_llarge} is given also for $N=n\geq 2$, where now $\pi$ 
in \eqref{valueA_llarge} is replaced by $\omega_n$.} that, 
for $\longR$ large enough,
\begin{align}\label{valueA_llarge}
 \relarea(u,\sourcedisk_\longR(0))
=|\mathcal G_u|+\pi=\int_{\sourcedisk_\longR(0)}\sqrt{1+|\nabla u|^2}dx+\pi.
\end{align}
With the aid of an example, they also show that $\relarea(u,
\sourcedisk_l(0))$ must be strictly smaller than the right-hand side
of \eqref{valueA_llarge}, 
since there is a sequence of $C^1$-maps 
approximating $u$ and having, asymptotically, a lower
value of $\area(\cdot, \Om)$. 
We anticipate here 
that, when $\longR$ is small, the above mentioned 
sequence is not optimal, and the construction 
of a recovery sequence for $\relarea(u,\sourcedisk_\longR)$ 
is much more involved and requires to solve a sort of Plateau-type 
problem in $\R^3$ with singular boundary, with a part of multiplicity $2$. This has been studied in \cite{BES2}, where with a reflection argument
with respect to a plane,  it can be seen as a non-parametric Plateau-type problem with 
a partial free boundary; in the special setting of \cite{BES2} it is possible to show that, excluding a singular configuration (corresponding to $l$ large), the solution is non-parametric and attains the zero boundary condition on the free part (we refer to \cite{BMS} for a more general setting where similar results are obtained). 

To state our main result we need to fix some 
notation.
For $\longR>0$ we denote $R_{2\longR}:=(0,2l)\times (-1,1)$
and let $\partial_DR_{2l}:=(\{0,2l\}\times[-1,1])\cup((0,2l)\times \{-1\})$ be
what we call the Dirichlet boundary  of $R_{2\longR}$.
Define  
$\varphi: \partial_D R_{2\longR} \rightarrow [0,1]$ as
$\varphi(t,s)
:= \sqrt{1-s^2}$ 
if $(t,s) \in \{0,2\longR\} \times [-1,1]$
and $\varphi(t,s):=0$ if $(t,s) \in (0,2\longR) \times \{-1\}$.
Let 
\begin{align*}
	&\widetilde {\mathcal H}_{2l}:=\{
	h:[0,2l]\rightarrow [-1,1], \ h ~{\rm continuous},~ h(0)=h(2l)=1
	\},
	\\
	&\mathcal X_{D,\varphi}:=\{\psi\in W^{1,1}(R_{2l}):\psi=\varphi\text{ on }\partial_DR_{2l}\},
\end{align*}
and for any $h\in \widetilde {\mathcal H}_{2l}$ set
$ G_h:=\{(t,s)\in R_{2l}:s=h(t)\}$ and $SG_h:=\{(t,s)\in R_{2l}:s\leq h(t)\}$. The 
main result of the present paper (see Theorem \ref{Thm:maintheorem}) reads as follows: 

\begin{theorem}\label{teo_intro}
	Let $N=n=2$, $l>0$ and $u:\sourcedisk_\longR(0)\rightarrow\R^2$ be the vortex map defined in \eqref{vortexmapdef}. Then 
	\begin{align}\label{value_main}
		\relarea (u,\sourcedisk_\longR(0))
		\leq \int_{\sourcedisk_\longR(0)}\sqrt{1+|\nabla u|^2}dx+\inf\{ \mathcal A(\psi,SG_h):(h,\psi)\in \widetilde {\mathcal H}_{2l}\times 
		\mathcal X_{D,\varphi},\;\psi=0\text{ on }G_h\}. 
	\end{align}
\end{theorem}

We emphasize that, for $l$ large, the infimum on the right-hand side is $\pi$. 
Further, thanks to the opposite inequality proved in \cite{BES1}, equality holds in Theorem \ref{teo_intro}, for each value of $l>0$.

For $l$ small, the fact that the 
sequence leading to the value 
in \eqref{valueA_llarge} is not optimal is strongly related with the choice of the $L^1$-convergence in the 
definition \eqref{area_relax} of $\overline{\mathcal A}(\cdot, \Omega)$. 
Even if this seems the most natural notion of convergence for the 
approximating maps $v_k$ of $u$, one can also opts to choose stronger topologies. Some results are known when one chooses, 
instead of the $L^1$-convergence, the strict convergence in $BV(\Om;\R^2)$ (see \cite{CM,Mucci,BCS,BCS2,C}). With this convergence, it has been shown in \cite{BCS} that the relaxed area of the vortex map $u$ always equals the right-hand side of \eqref{valueA_llarge}.

In order to give an idea of how
the value $\pi$ in \eqref{valueA_llarge} pops up (and then how it 
appears in \eqref{value_main} for $l$ large), it is convenient to introduce the tool of Cartesian currents.
One can regard the graphs $G_v=\{(x,y)\in \Omega\times \R^2:y=v(x)\}$ of a $C^1$ map $v:\Omega\rightarrow\R^2$ as an integer multiplicity 
$2$-current in $\Omega\times \R^2$. It is seen that a 
sequence $(G_{u_k})$ with $u_k$ 
approaching $u$ and with $\sup_k \area(u_k,\Om)<+\infty$, 
converges\footnote{This is a consequence of Federer-Fleming closure theorem.}, up to subsequences, to a Cartesian current $T$ which splits as $T=\mathcal G_u+S$, with $S$ a vertical integral current such that $\partial S=-\partial \mathcal G_u$.
A direct computation shows that 
$$\partial \mathcal G_u=-\delta_0\times \partial \jump{B_1}$$ 
(see \cite[Section 3.2.2]{GiMoSu:98}),
so that the problem of determining the value of $\relarea(u,\Omega)$ is somehow related to the 
computation of the mass of a {mass-minimizing} vertical current $S_\textrm{min}\in \mathcal D_2(\Omega\times \R^2)$ satisfying 
\begin{align}\label{boundary_of_u}
 \partial S_\textrm{min}=\delta_0\times \partial \jump{B_1}\qquad \text{in }\mathcal D_1(\Omega\times \R^2).
\end{align}
In some cases, and in particular for $\longR$ large, these two problems are related, and it turns out that $S_{\rm min}
=\delta_0\times\jump{B_1}$, whose mass is $\pi$. However $S_\textrm{min}\neq \delta_0\times\jump{B_1}$ for $\longR$ small. Moreover, the two problems of determining $S_\textrm{min}$ and the value of the relaxed area 
functional are, unfortunately, not
 related in general. This is mainly due  to the following two obstructions:
\begin{itemize}
 \item we have to guarantee that the current $\mathcal G_u+S_\textrm{min}$ is obtained as a limit of smooth graphs, that is not easy to establish, since not all Cartesian currents can be obtained as such limits (see \cite[Section 4.2.2]{GiMoSu:98});
 \item even if $\mathcal G_u+S_\textrm{min}$ is 
limit of graphs $\mathcal G_{u_k}$ of smooth maps $u_k$, nothing ensures 
that $\area(u_k,\Omega)\rightarrow\relarea(u,\Omega)$, 
due to possible cancellations of the currents $\mathcal G_{u_k}$ that, 
in the limit, might overlap with opposite orientation.
\end{itemize}

Actually, in many cases, as in the one considered in this paper, 
for an optimal sequence $(u_k)$ 
realizing the value of $\relarea(u,\Omega)$, it holds
\begin{align}\label{Sopt}
\mathcal G_{u_k}\rightharpoonup\mathcal G_u+S_{\textrm{opt}}\neq \mathcal G_u+S_{\textrm{min}},
\end{align}
and the limit vertical part $S_{\textrm{opt}}$ 
satisfies $|S_{\textrm{opt}}|>|S_{\textrm{min}}|$. 
For instance, if $l$ is small, 
it is possible to construct a sequence $(\widehat u_k)$ 
approaching $u$ which is not a recovery sequence 
for $\relarea(u,\Om)$, but whose limit vertical part $S_{\textrm{min}}$ has mass strictly smaller than the mass of $S_{\textrm{opt}}$ (see Section \ref{sec:catenoid_with_a_flap}). In this case, a 
suitable projection of $S_\textrm{min}$ in $\mathbb R^3$ is half of a classical area-minimizing catenoid between two unit circles at distance $2l$ from each other.

An additional source of difficulties 
in the computation of $\relarea(u, \Om)$ is due to an 
example \cite{Scala:19}
valid for the triple junction map $u_T$, 
 and showing that the equality
\begin{align}\label{triplevalue}
 \relarea(u_T,\Omega)=|\mathcal G_{u_T}|+3m_\longR
\end{align}
holds only under some additional requirements; for instance if the 
triple junction point is exactly located at
the origin $0$ and the domain is a 
disc $\Omega=\sourcedisk_\longR(0)$ around it. In particular, for different domains, \eqref{triplevalue} is no longer valid, and $S_\textrm{opt}$ is a vertical current whose support projection on $\Omega$ is a set connecting the triple point with $\partial\Omega$, and which does not coincide
 with (neither is a subset of) the jump set of $u_T$ (see 
\cite[Example in Section 6]{Scala:19} and also \cite{BeElPaSc:19} for other non-symmetric settings).

A similar behaviour of the vertical part $S_\textrm{opt}$ 
holds for $u$: when $\longR$ is small, 
the projection of 
$S_\textrm{opt}$ on $\sourcedisk_\longR(0)$ concentrates over a radius connecting $0$ to $\partial \sourcedisk_\longR(0)$. However, if the domain $\Omega$ loses its symmetry, 
almost nothing is known about $S_\textrm{opt}$.

This kind of phenomena have been observed also in 
other cases, as in \cite{BePaTe:15,BePaTe:16} where $BV$-maps  $u:\Omega\rightarrow\R^2$
with a prescribed discontinuity on a curve (jump set) are considered. The creation of such ``phantom bridges'' 
between the singularities of the 
map $u$ and the boundary of the domain is very specific of the choice of the $L^1$-convergence 
in the computation of $\relarea(\cdot, \Omega)$. As already said, 
other choices are possible, giving rise to different relaxed 
functionals\footnote{Relaxing $\area(\cdot,\Om)$ in stronger
	topologies $\tau$ is possible (see, e.g., \cite{Mucci,BCS}); however, this would make more
	difficult to prove, eventually, $\tau$-coercivity of 
$\relarea(\cdot,\Omega)$. In 
	addition, it could destroy the interesting
	nonlocal phenomena related to the 
	appearence of certain nonstandard Plateau problems,
	which are the focus of this paper.} 
\cite{BePaTe:15,BePaTe:16}.  

The nonlocality and the uncontrollability of $S_\textrm{opt}$ are more and more evident if we try to generalize \eqref{triplevalue} dropping the assumption that the range of $u_T$ 
consists of the vertices of an equilateral triangle. If we assume that $u_T$ takes values in $\{\alpha,\beta,\gamma\}$, three generic (not aligned)
points in $\R^2$ then, also if the domain of $u_T$ is symmetric, there is no 
sharp computation of $\relarea(u_T,\Omega)$.
In this case, the analysis  is related to an entangled Plateau problem, where three area-minimizing discs have as partial free boundary three curves connecting the couples of points in $\{\alpha,\beta,\gamma\}$, respectively, and where these three curves are forced to overlap. Some  partial results had been obtained in \cite{BeElPaSc:19}, where the authors find an upper bound for $\relarea(u_T,\Omega)$. However the question of finding the value of $\relarea(\cdot, \Omega)$ for this piecewise constant maps seems to be 
difficult. In the case that $u$ is piecewise costant and takes three values vertices of an equilateral triangle, as for the triple junction map but in general domains, some upper bounds have been provided \cite{ScSc}. The singular contribution of the area is related with the flat norm of the distributional Jacobian of such maps \cite{DLSVG}. Similarly, when $u\in W^{1,1}(\Om;\mathbb S^1)$ is circle-valued, 
it is possible to show that the singular contribution of the area is bounded above by (a suitable multiple) of the flat norm of $\det (\nabla u)$ (see \cite{BSS}).
\medskip

Let us go back to the minimum problem 
\begin{align}\label{mio2}
\inf\{ \mathcal A(\psi,SG_h):(h,\psi)\in \widetilde {\mathcal H}_{2l}\times 
	\mathcal X_{D,\varphi},\;\psi=0\text{ on }G_h\}.
\end{align}
Following \cite{BES2}, this problem has many formulations and 
it is proved that the infimum in 
\eqref{mio2} coincides with 
\begin{align}\label{mio}
\min\big \{ \FB(\h,\psione):  (\h,\psione) \in  
\domalaa
\big \}.
\end{align}
Here we refer to Section \ref{sec:setting_of_the_problem} for the notation and definition of $\FB$. Also, a solution to this minimum problem 
has been proven to exist and satisfies suitable regularity property if $l\leq l_0$, for some threshold $l_0>0$
(see Theorem \ref{Thm:existenceofminimizer}). 
If instead $l>l_0$, the unique solution to \eqref{mio}  is given by the two constants maps $h\equiv1$ and $\psi\equiv0$, corresponding to the case where $\FB$ measures the area of two half-discs of radius $1$, namely providing the value $\pi$ appearing in \eqref{valueA_llarge}.

We do not know the explicit value of the threshold $l_0$.
However, it is clear that $l_0>\frac12$ 
(see \cite{BES2}).
Furthermore, let the surface $\Sigma^+$ be the graph of a regular solution $\psi$ when $l<l_0$. Doubling the surface $\Sigma^+$ 
by considering 
its symmetric with respect to the plane containing $R_{2l}$, and then 
taking the union $\Sigma$ of these two area-minimizing surfaces, 
it turns out that $\Sigma$ solves a non-standard Plateau problem, spanning a nonsimple curve which shows self-intersections 
(this is the union of $\Gamma$ with its symmetric with 
respect to $R_{2l}$, the obtained curve is the union of two circles connected by a segment \cite{BES2}). 
Again, the obtained
area-minimizing surface is a sort of catenoid forced to contain a segment
 for $\longR$ small, and two distinct discs spanning the two circles for $\longR$ large. 
The restriction of $\Sigma$ to the set $\overline B_1\times [0,l]$ is a suitable projection in $\R^3$ of the aforementioned vertical current $S_{\rm opt}$.

In order to prove our main result, the analysis 
consists in a careful definition of a recovery sequence $(u_k)$
converging to the vortex map, and thus such that 
$\area(u_k,
\sourcedisk_\longR(0))$
approaches the value on the right-hand side of \eqref{value_main}
as $k \to +\infty$. 
To explicitely construct $u_k$, we need first to relate the minimum problem stated in \eqref{mio} with the non-parametric Plateau-type problem in \eqref{mio2}; this is obtained in \cite{BES2}, where we exploit the convexity of the domain together with some well-known regularity results for the solution of the Plateau problem in this setting. This analysis leads us to Theorem \ref{Thm:existenceofminimizer}, which 
characterizes the solution of \eqref{mio}, and which is based on a regularity result for the minimizing pair $(h^\star,\psi^\star)\in \widetilde {\mathcal H}_{2l}\times 
\mathcal X_{D,\varphi}$.  Finally, thanks to the regularity 
results that we have obtained (especially, boundary
regularity), in Section \ref{sec:final} we define explicitely the 
maps $u_k$,
making a crucial use of rescaled versions 
of the area-minimizing surface $\Sigma$ in a vertical copy of $\R^3$ inside
$\R^4$,
 and prove the upper bound in Theorem \ref{Thm:maintheorem}.

The paper is organized as follows: in Sections \ref{sec:preliminaries} 
and \ref{sec:setting_of_the_problem} we introduce some notation and the setting of the problem. 
In Section \ref{sec:three_examples} 
we provide some examples of potential recovery sequences, 
one of which is optimal in the case $l$ large. 
Finally, in Section \ref{sec:final} we construct a
 recovery sequence in the more involved case $l\leq l_0$. 

%%%%%%%%%%%%%%%%%%%%%%%%%%%%%%%%%%%%%%%%%%%%%%%%%%%%%%%%%%%%%%%%%%%%%%%%%
\section{Preliminaries}\label{sec:preliminaries}
%%%%%%%%%%%%%%%%%%%%%%%%%%%%%%%%%%%%%%%%%%%%%%%%%%%%%%%%%%%%%%%%%%%%%%%%%

The symbol $\area(v,\Omega)$ denotes  the classical area of the graph of a smooth map $v:\Omega\subset\R^n\rightarrow \R^N$, 
given by the right hand side of \eqref{area_classical}.  We will deal with the case $n=2$ and mostly with the cases $(n,N)=(2,1)$ and 
$(n,N)=(2,2)$. The $L^1$-relaxed area functional 
is denoted by $\relarea(v,\Omega)$ and is defined in \eqref{area_relax}.

We first remark that the infimum in \eqref{area_relax} can be equivalently
considered as taken over the class of sequences $(v_k)\subset \textrm{Lip}(\Omega;\R^2)$. 
This does not change the value of $\relarea(\cdot,\Om)$, as observed in \cite{BePa:10}. 

Recall that in formula \eqref{area_classical} the symbol $\mathcal M(\nabla v)$ denotes the vector whose entries are all 
determinants of the minors of $\nabla v$. Precisely, let $\alpha$ and $\beta$ be subsets of $\{1,2\}$, let $\bar\alpha$ denote the complementary set of $\alpha$, namely $\bar\alpha=\{1,2\}\setminus\alpha$, let $|\cdot|$ denote the cardinality, and let $A\in \R^{2\times2}$ be a matrix. Then, if $|\alpha|+|\beta|=2$, we denote by  
$M_{\bar\alpha}^\beta(A)$
 the determinant of the submatrix of $A$ whose lines are those with index in $\beta$, and columns with index in $\bar\alpha$. By convention $M_\emptyset^\emptyset(A)=1$ and  moreover
$$M_{j}^i=a_{ij},\qquad i,j\in \{1,2\},\qquad \qquad M_{\{1,2\}}^{\{1,2\}}(A)=\det A,$$
and the vector $\mathcal M(A)$ 
takes the form $$\mathcal M(A)=(M_{\bar\alpha}^\beta)(A)=(1,a_{11},a_{12},a_{21},a_{22},\det A),$$
where $\alpha$ and $\beta$ run over all the subsets of $\{1,2\}$ with the constraint $|\alpha|+|\beta|=2$. We identify $\alpha$ and $\beta$ as multi-indices in $\{1,2\}$.

\subsubsection{Area in cylindrical coordinates}
Polar coordinates
in the source space $\R^2_{{\rm source}}$ are denoted by $(\sourceradialcoordinate,
\sourceangularcoordinate)$.
Polar coordinates
in the target space $\R^2_{{\rm target}}$ are denoted by $(\rho,\theta)$. 

Assume that $B=\{(r,\sourceangularcoordinate)\in \R^2: r \in (r_0,r_1),
~\sourceangularcoordinate \in (\sourceangularcoordinate_0,\sourceangularcoordinate_1)\}$; then the area of the graph of 
the smooth map $v=(v_1,v_2)$ in polar coordinates over $B$ is given by 
 \begin{equation*}
  \area (v,B)=\int_{r_0}^{r_1}\int_{\sourceangularcoordinate_0}^{\sourceangularcoordinate_1} \vert \M(\nabla v)\vert (r,\sourceangularcoordinate)~ r dr d\sourceangularcoordinate.
 \end{equation*}
Recall that, for $i \in \{1,2\}$, we have 
\begin{align*}
\partial_{x_1}v_i=\cos \sourceangularcoordinate
 \partial_r v_i - \frac{1}{r} \sin \sourceangularcoordinate
  \partial_\sourceangularcoordinate v_i,\qquad 
\partial_{x_2}v_i=\sin\sourceangularcoordinate
 \partial_r v_i + \frac{1}{r} \cos \sourceangularcoordinate
 \partial_\sourceangularcoordinate v_i.
\end{align*}
Hence 
\begin{align}
&\vert \nabla v_i\vert ^2 = 
(\partial _r v_i)^2 +\frac{1}{r^2} (\partial_\sourceangularcoordinate v_i)^2,\qquad i \in \{1,2\},\label{eqn:gradinpolar}
\\
&\partial_ {x_1} v_1 \partial_{ x_2} v_2-\partial_{ x_2}v_1
\partial_{ x_1}v_2 =\frac{1}{r}\Big( 
\partial_ {r} v_1 \partial_{\sourceangularcoordinate} v_2-\partial_{\sourceangularcoordinate}v_1
\partial_{ r}v_2\Big).
\nonumber
\end{align}

Thus the area of the graph of $v$ over $B$ is given by 
\begin{equation}\label{eqn:areapolarexpression}
\begin{aligned}
& \area (v,B)
\\
=& \int_{r_0}^{r_1}\int_{\sourceangularcoordinate_0}^{\sourceangularcoordinate_1} 
\sqrt{1+(\partial_r v_1)^2+(\partial_r v_2)^2+\frac{1}{r^2}
\left\{
(\partial_\sourceangularcoordinate
 v_1)^2+
(\partial_\sourceangularcoordinate
 v_2)^2+\Big( 
\partial_ {r} v_1 \partial_{\sourceangularcoordinate} v_2-\partial_{\sourceangularcoordinate}v_1
\partial_{ r}v_2\Big)^2\right\}}~ r dr d\sourceangularcoordinate.
\end{aligned}
\end{equation}
%
%%%%%%%%%%%%%%%%%%%%%%%%%%%%%%%%%%%%%%%%%%%%%%%%%%%%%%%%%%%%%%%%%%%%%%%%%
%%%%%%%%%%%%%%%%%%%%%%%%%%%%%%%%%%%%%%%%%%%%%%%%%%%%%%%%%%%%%%%%%%%%%%%%%
We denote by $\sourcedisk_r=\sourcedisk_r(0)\subset\R^2 = \R^2_{{\rm source}}$ the open disc centered at $0$ 
with radius $r>0$ in the source space. 
Our reference domain is $\Omega=\sourcedisk_\longR \subset \R^2_{{\rm source}}= \R^2_{(x_1,x_2)}$
where $l>0$ is fixed once for all. 

%%%%%%%%%%%%%%%%%%%%%%%%%%%%%%%%%%%%%%%%%%%%%%%%%%%%%%%%%%%%%%%%%%%%%%%%
\subsection{Graphs in codimension $1$}\label{sec:generalized_graphs_in_codimension_one}
%%%%%%%%%%%%%%%%%%%%%%%%%%%%%%%%%%%%%%%%%%%%%%%%%%%%%%%%%%%%%%%%%%%%%%%%
Let $\Om$ be an open bounded set and let $v\in L^1(\Omega)$. If $v\in C^1(\Om)$ 
the classical area of its graph is given by 
$$\mathcal A(v,\Om):=\int_\Om\sqrt{1+|\nabla v|^2}dx.$$
This notion is extended to every function $v\in L^1(\Omega)$ by relaxation as 
in \eqref{area_relax}, and $\overline{\mathcal  A}(v,\Om)$ coincides with \eqref{integral_formula}. For all $v\in L^1(\Om)$ we denote by $R_v\subseteq\Omega$ the set of regular points of $v$, 
{\it i.e.}, the set consisting of points $x$ which are Lebesgue points for 
$v$, $v(x)$ 
coincides with the Lebesgue value of $v$ at $x$, and $v$ 
is approximately differentiable at $x$. We also set
\begin{align*}
 &G_v^R:=\{(x,v(x))\in R_v\times \R\},
\\
 &SG_v^R:=\{(x,y)\in R_v\times \R:y< v(x)\}.
 \end{align*}
 We often will identify $SG_v^R$ with the integral $3$-current $\jump{SG_v}\in \mathcal D_3(\Omega\times\R)$. If $v$ is a 
function of bounded variation, $\Omega\setminus R_v$ has 
zero Lebesgue measure, so that the current $ \jump{SG_v}$ 
coincides with the integration over the subgraph 
 \begin{align*}
 &SG_v:=\{(x,y)\in \Omega\times \R:y< v(x)\}.
 \end{align*}
For this reason we often identify $SG_v=SG_v^R$.
It is well-known that the  perimeter of $SG_v$ in $\Omega\times \R$ 
coincides with $\relarea(v,\Omega)$.

The support of the boundary of $\jump{SG_v}$ 
includes the graph $G_v^R$, but in general consists also of additional parts, called vertical. We denote by
$$\mathcal G_v:=\partial \jump{SG_v}\res(\Omega\times\R),$$
the generalized graph of $v$, which is a $2$-integral current supported on $\partial^*SG_v$, the reduced boundary of $SG_v$ in $\Omega\times\R$. 

Let $\widehat \Omega \subset \R^2$ 
be a bounded  open set such that $\Omega\subseteq \widehat\Omega$, 
and suppose that 
$L := \widehat \Omega\cap
\partial \Omega$ is a rectifiable 
curve. Given $\scalarfunction 
\in BV(\Omega)$ and  a $W^{1,1}$ function $\varphi:
\widehat \Omega \to \R$, we can consider 
\begin{align*}
 \overline{\scalarfunction}:=\begin{cases}
               f & \text{on }\Omega,
\\
               \varphi&\text{on }\widehat \Omega\setminus\Omega.
               \end{cases}
\end{align*}
Then (see \cite{Giusti:84}, \cite{AmFuPa:00})
\begin{align*}
 \relarea(\overline{\scalarfunction},\widehat\Om)=
\relarea(\scalarfunction,\Om)+\int_{L}|\scalarfunction-\varphi|d\mathcal H^{1}
+\relarea(\varphi,\widehat\Om
\setminus\overline{\Om}).
\end{align*}
%
%%%%%%%%%%%%%%%%%%%%%%%%%%%%%%%%%%%%%%%%%%%%%%%%%%%%%%%%%%%%%
\section{Setting of the problem}\label{sec:setting_of_the_problem}

Let us focus on the minimum problem on the right hand side of \eqref{value_main}, i.e., 
\begin{align}
	\inf\{ \mathcal A(\psi,SG_h):(h,\psi)\in \widetilde {\mathcal H}_{2l}\times 
	\mathcal X_{D,\varphi},\;\psi=0\text{ on }G_h\}.
\end{align}

\begin{definition}[\textbf{The functional $\FB$}] \label{def:FB}
	We define 
	\begin{align}\label{eq:space_of_h_in_the_doubled_interval}
		&\domalaa
		:=
		\left\{(h,\psi): h \in \Hspace, \psi\in BV(\doubledrectangle, [0,1]),\psi = 0 \ 
		{\rm on}~ \OmAlaa \setminus \subgraphh \right\},\\
		&
		\Hspace
		=\big \{\h : [0,2 \longR] \to [-1,1], 
		~ \h {\rm~ convex}, ~
		\h(\axialcoordofcylinder)=\h(2\longR-\axialcoordofcylinder) 
		~\forall \axialcoordofcylinder \in [0,2\longR]
		\big \}.
	\end{align}
	and for any $(h,\psi) \in 
	\domalaa$, 
	\begin{equation}
		\FB(\h,\psione):=\mathcal A(\psione;\OmAlaa)-
		\Htwo(\OmAlaa \setminus \Omegah)+\int_{\partialbar\OmAlaa}  
		\vert \psione  -\dirdatum\vert d \Hone 
		+ \int_{
			\partial \OmAlaa \setminus \partialbar\OmAlaa
		}  \vert \psione \vert 
		~d \Hone.
		\label{eqn:FB}
	\end{equation}
\end{definition}
%%%%%%%%%%%%%%%%%%%%%%%%%%%%%%%%%%%%%%%%%%%%%%%%%%%%%%%%%%%%%

In \cite[Theorem 1.2]{BES2} it is shown that 
\begin{align}\label{43}
	\inf\{ \mathcal A(\psi,SG_h):(h,\psi)\in \widetilde {\mathcal H}_{2l}\times 
	\mathcal X_{D,\varphi},\;\psi=0\text{ on }G_h\}=\inf \big \{ \FB(\h,\psione):  (\h,\psione) \in  \domalaa
	\big \},
\end{align}
and for this reason it is necessary to investigate existence and regularity of minimizers of $ \FB$.
To this aim it is first convenient to extend $\varphi$ in the doubled rectangle
$\overline R_{2l}$ by defining the extension $\widehat \varphi$
as:
\begin{equation}\label{eq:widehat_varphi}
	\widehat \varphi(w_1,w_2)=
	\widehat 
	\varphi(0,w_2):=\sqrt{1-w_2^2}\qquad \forall (w_1,w_2)\in \overline R_{2l}.
\end{equation}
From \cite[Theorem 1.1]{BES2} the following result follows:

\begin{theorem}[\textbf{Minimizing pairs}]
	\label{Thm:existenceofminimizer}
	There exists $(\hstar,\psionestar) \in  \domalaa$  such that 
	\begin{equation}\label{eqn:B}
		\FB(\hstar, \psionestar)=\min \big \{ \FB(\h,\psione):  (\h,\psione) \in  \domalaa
		\big \},
	\end{equation}
	and $\psionestar$ is symmetric with respect to 
	$\{\axialcoordofcylinder=\longR\} \cap R_{2l}$. Moreover there exists a threshold $l_0>0$ such that, for $l>l_0$ the above minimizer is $(h^\star,\psi^\star)=(1,0)$, two constant functions, whereas for $0 <l\leq l_0$ the above minimizer satisfies the following features: $\hstar$ is not identically $-1$ and 
	\begin{itemize}
		\item[(i)] $\hstar(0)=1=\hstar(2\longR)$, and 
		$\hstar >-1$ in $(0,2\longR)$;
		\item[(ii)] $\psionestar$ is locally Lipschitz,  analytic,
		and strictly positive
		in $\Omegahstar$; 
		\item[(iii)] $\psionestar$ 
		is continuous up to the boundary of $\Omegahstar$, 
		and attains the boundary conditions, \textit{i.e.}, for $(\tcoord, \scoord)
		\in 
		\partial SG_{h^*}$,  
		\begin{equation}\label{eqn:psionestarboundary}
			\psionestar(\axialcoordofcylinder,\scoord)=\begin{cases}
				0& \text{ if}\quad \scoord=-1 \text{ or }\scoord=\hstar(\axialcoordofcylinder),\\
				\sqrt{1-\scoord^2}&\text{ if}\quad \axialcoordofcylinder=0 \text{ or } \axialcoordofcylinder=2\longR,
			\end{cases} 
		\end{equation}
		hence 
		\begin{equation}\label{eqn:FAequalareaofhstar}
			\FB(\hstar,\psionestar)=\mathcal A(\psionestar, \Omegahstar);
		\end{equation}
		\item[(iv)] we have 
		\begin{equation}\label{eq:psi_star_below_cylinder}
			\psi^\star < \widehat \varphi {\rm ~in}~ R_{2l}.
		\end{equation}
	\end{itemize}
\end{theorem}

A minimizer $(h^\star, \psi^\star)$ of \eqref{eqn:B}
is needed
for constructing a recovery sequence $(u_k) \subset {\rm Lip}(\Om, \R^2)$,
see formulas \eqref{vepsonCeps} and \eqref{vepsonCeps-}:
we know that 
$\psi^\star$ 
is locally Lipschitz, but not Lipschitz, in $R_{2l}$, therefore we need first a regularization 
procedure. This is made in Lemma \ref{lem:properties_of_psi_m} below, that will be 
used in the proof of step 2 of Theorem \ref{Thm:maintheorem}.

Let $(h^\star,\psi^\star)$ be a minimizer provided by 
Theorem \ref{Thm:existenceofminimizer}, and assume that $h^\star$ is not identically $-1$ (namely, we are in the case $l\leq l_0)$. 
We fix an integer $m>0$
and, recalling the definition of $\widehat \varphi$ in 
\eqref{eq:widehat_varphi}, define
\begin{align}\label{def_varphi_m}
\varphi_m:=\Big(\widehat 
\varphi-\frac2m\Big)\vee0 
\qquad {\rm in}~ \overline R_{2l}.
\end{align}
We observe that  $\varphi_m$ is Lipschitz continuous in $\overline R_{2l}$.
We then set
\begin{align}\label{def_psi_m}
\psi^\star_m:=\Big(\big(\psi^\star-\frac1m)\vee0\big)\Big)\wedge\varphi_m
\qquad {\rm in}~ R_{2l}.
\end{align}
Since $\psi^\star$ is locally Lipschitz in $R_{2l}$, an easy check
shows that $\psi^\star_m$ is Lipschitz continuous in $R_{2l}$ for
any $m$ (with an unbounded Lipschitz constant as $m\rightarrow+\infty$).
 This follows from the fact that 
$\psi^\star$ is continuous up to the boundary of $R_{2l}$
(see Theorem  \ref{Thm:existenceofminimizer} (iii)) 
and  hence 
$\psi^\star_m$ coincides with 
either $0$ or $\varphi_m$ in a 
neighborhood of $(\partial_D R_{2l})\cup G_{h^\star}$ in $R_{2\longR}$. 
Furthermore, still $\psi^\star_m=0$ on the upper graph $\overline{R_{2l}}\setminus SG_{h^\star}=\{(w_1,w_2)\in \overline{R_{2l}}:w_2\geq h^\star(w_1)\}$ of $h^\star$.

\begin{lemma}[\textbf{Properties of $\psi_m^\star$}]
\label{lem:properties_of_psi_m}
	Let $(h^\star,\psi^\star)$ be a minimizer of $\FB$ as in 
Theorem \ref{Thm:existenceofminimizer} and assume $h^\star$ is not identically $-1$. 
For all $m>0$ let $\psi^\star_m$ be defined as in \eqref{def_psi_m}.
	Then: 
	\begin{itemize}
		\item[(i)] $\psi^\star_m$ is 
Lipschitz continuous in $\overline{SG_{h^\star}}$,  $\psi^\star_m=0$ 
on $([0,2l]\times\{-1\}) \cup (\overline{R_{2l}}\setminus SG_{h^\star})$, 
and $\psi^\star_m(0,\cdot)=\varphi_m(0,\cdot)$, so that  $|\partial_{w_2}\psi^\star_m(0,\cdot)|\leq|\partial_{w_2}\varphi(0,\cdot)|=|\partial_{w_2}\psi^\star(0,\cdot)|$ a.e. in $[-1,1]$; \
		
		\item[(ii)] $(\psi^\star_m)$ 
converges to $\psi^\star$ uniformly on $\{0,2l\}\times [-1,1]$ as $m \to +\infty$;
		
		\item[(iii)] we have 
\begin{align}\label{eq:lim_m}
		\lim_{m\rightarrow +\infty} 
\mathcal A(\psi^\star_m, SG_{h^\star})=
		\mathcal A(\psi^\star, SG_{h^\star}).
		\end{align}
\end{itemize}
	As a consequence $\FB(h^\star,\psi^\star_m)\rightarrow \FB(h^\star,\psi^\star)$ as $m\rightarrow +\infty$.
\end{lemma}
\begin{proof} (i) and (ii) are direct consequences of the definitions. 
To show (iii) 
	we start to observe that 
$\psi^\star_m\rightarrow \psi^\star$ 
pointwise in $R_{2l}$: indeed, 
this follows from the definitions 
of $\varphi^\star_m$ and $\psi^\star_m$ up to noticing that 
$\varphi_m\rightarrow \widehat \varphi$ pointwise in $R_{2l}$ 
as $m\rightarrow +\infty$, 
and $\psi^\star\leq \widehat 
\varphi$ on $R_{2l}$. 
From Theorem \ref{Thm:existenceofminimizer} (iv) it follows that, 
at any point $(w_1,w_2)\in R_{2l}$, 
for $m$ large enough 
$\varphi_m(w_1,w_2)>\psi^\star(w_1,w_2)$ 
(since $\widehat \varphi(w_1,w_2)>\psi^\star(w_1,w_2)$), 
so that 
$\psi^\star_m(w_1,w_2)=\psi^\star(w_1,w_2)-\frac1m$. 
	As a consequence the set $A_m:=\{0<\psi^\star-\frac1m<
	\varphi_m\}$ satisfies 
$$
\lim_{m \to +\infty} \mathcal H^2(SG_{h^\star}\setminus A_m)=0,
$$
and on $A_m$ it holds
 $\psi^\star_m=\psi^\star-\frac1m$ and  $\nabla \psi^\star_m=\nabla \psi^\star$. Moreover,
 on $SG_{h^\star}\setminus A_m$, 
either $\psi^\star_m=0$ (and hence $\nabla \psi_m^\star=0$) or $\psi^\star_m=\varphi_m$ (and hence $\nabla \psi^\star_m=\nabla 
\varphi_m$). Therefore
$$
\int_{SG_{h^\star}\setminus A_m}\sqrt{1+|\nabla\psi^\star_m|^2}~dx
\leq \int_{SG_{h^\star}\setminus A_m}\sqrt{1+|\nabla\varphi_m|^2}~dx
$$
and 
$$
\lim_{m \to +\infty}
\int_{SG_{h^\star}\setminus A_m}\sqrt{1+|\nabla\psi^\star_m|^2}~dx
\leq \lim_{m \to +\infty}
\int_{SG_{h^\star}\setminus A_m}\sqrt{1+|\nabla\varphi_m|^2}~dx
= 0,$$
because $\vert \nabla \varphi_m\vert$ are 
uniformly bounded in $L^1(R_{2l})$.
Also

	\begin{align*}
	\mathcal A(\psi^\star_m, SG_{h^\star})=\int_{A_m}\sqrt{1+|\nabla\psi^\star|^2}~dx+\int_{SG_{h^\star}\setminus A_m}\sqrt{1+|\nabla\psi^\star_m|^2}~dx,
	\end{align*}
and \eqref{eq:lim_m} follows.
\end{proof}

The main result of this paper reads
as follows.

\begin{theorem}[\textbf{Upper bound for the area of the vortex map}]
\label{Thm:maintheorem}
The relaxed area of the graph of the vortex map $\vmap$ satisfies 
\begin{equation}
\label{eq:upper_bound_recovery}
\relarea(\vmap,\BallR)\leq 
\int_{\BallR}\vert \M(\nabla \vmap)\vert ~dx  + \inf \big \{ \FB(\h,\psione):  (\h,\psione) \in  \domalaa
\big \}.
\end{equation}
\end{theorem}
Notice that by \eqref{43} this result is equivalent to Theorem \ref{teo_intro}.

\section{Some examples}\label{sec:three_examples}
%%%%%%%%%%%%%%%%%%%%%%%%%%%%%%%%%%%%%%%%%%%%%%%%%%%%%%%%%%%%%%%%%%%%%%%
Before going into the details of Theorem \ref{Thm:maintheorem},
it is worth making some nontrivial examples, which are also
useful for the understanding of the proof of the theorem.

\subsection{An approximating sequence of maps with degree zero: cylinder} 
\label{subsec:an_approximating_sequence_of_maps_with_degree_zero:cylinder}
In \cite{AcDa:94} the authors describe 
a sequence $(u_k)$ of Lipschitz maps
converging to $u$ and taking values
in $\mathbb S^1$; 
in our context,
$u_k$ is defined in polar coordinates as follows:
\begin{equation}
	\label{eqn:u_k_Cylinder}
	\veps(r,\alpha):=
	\begin{cases}
		u(r, \alpha) = (\cos \alpha, \sin \alpha) & \textrm{ \qquad 
			in } \Omega_1:=\BallR \setminus (\Balleps \cup \{\alpha \in (-\alpha_k, \alpha_k)\}),\\
		(\cos (\frac{r}{\reps}(\alpha -\pi)+\pi),\sin (\frac{r}{\reps}(\alpha -\pi)+\pi)) & \textrm{ \qquad in }\Balleps \setminus \{\alpha \in (-\alpha_k, \alpha_k)\},\\
		(\cos (\frac{\alpha_k -\pi}{\alpha_k}\alpha +\pi),\sin(\frac{\alpha_k -\pi}{\alpha_k}\alpha +\pi)) & 
		\textrm{ \qquad in } \{\alpha \in [0, \alpha_k)\}\setminus \Balleps,\\
		(\cos (\frac{-\alpha_k +\pi}{-\alpha_k}\alpha +\pi),\sin(\frac{-\alpha_k +\pi}{-\alpha_k}\alpha +\pi)) & 
		\textrm{ \qquad in } \{\alpha \in (-\alpha_k,0)\}\setminus \Balleps,\\
		(\cos( \frac{r}{\reps}(\frac{\alpha_k -\pi}{\alpha_k}\alpha )+\pi),\sin ( \frac{r}{\reps}(\frac{\alpha_k -\pi}{\alpha_k}\alpha )+\pi)) & \textrm{ \qquad in }\Balleps \cap \{\alpha \in [0, \alpha_k)\},\\
		(\cos( \frac{r}{\reps}(\frac{-\alpha_k +\pi}{-\alpha_k}\alpha )+\pi),\sin ( \frac{r}{\reps}(\frac{-\alpha_k +\pi}{-\alpha_k}\alpha )+\pi)) & \textrm{ \qquad in }\Balleps \cap \{\alpha \in (-\alpha_k, \alpha_k)\},
	\end{cases} 
\end{equation}
where $(\reps)$ and $(\alpha_k)$ are two infinitesimal sequences 
of positive numbers; see Fig. \ref{fig:Omega_l_uk}.
Notice that
$u_k(0,0)= (-1,0)
=u_k(r,0)$ for $r \in (0,\longR)$.
Moreover for $t \in (0,\longR)$  we have  $u_k(\partial \sourcedisk_t)
=\partial B_1 \setminus \{\alpha \in (-\alpha_k, \alpha_k) \}$, and the degree
of $u_k$ is zero.

\begin{figure}
	\begin{center}
		\includegraphics[width=0.9\textwidth]{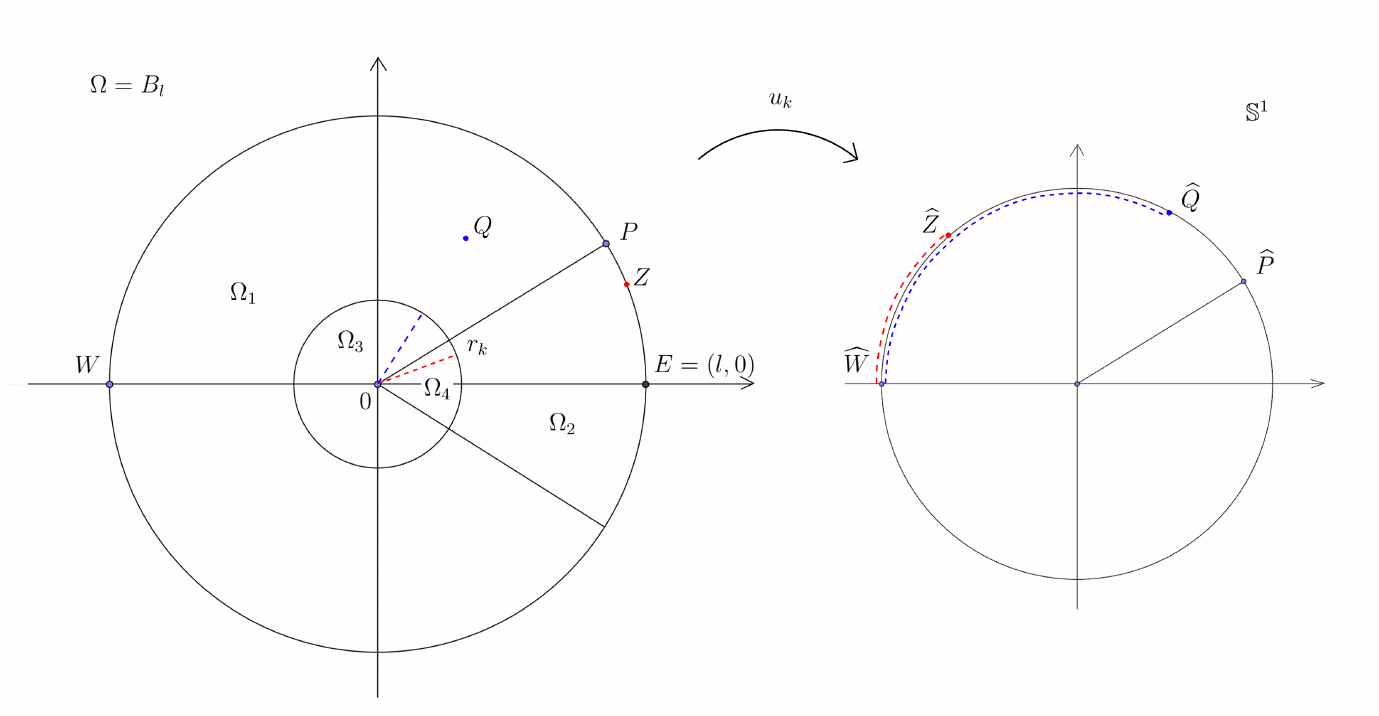}
		\caption{
			The map $u_k$ in \eqref{eqn:u_k_Cylinder}. 
			We set $\hat P := P/\vert P\vert = \alpha_k$, 
			$\hat Q := Q/\vert Q\vert$, 
			$\hat Z := Z/\vert Z\vert$, 
			$\hat W := W/\vert W\vert$. All points in $\Omega_1 \cup \Omega_2$ are retracted
			on $\mathbb S^1$ and suitably interpolated. 			The image of $\Omega_3$ through $u_k$ is as follows:
			$u_k$ sends the generic dotted segment onto the 
			(long) dotted arc on $\mathbb S^1$.
			Finally, the image of $\Omega_4$ through $u_k$ is as follows:
			$u_k$ sends the generic dotted segment onto the 
			(short) dotted arc on $\mathbb S^1$: Thus a short arc centered
at $E/\vert E\vert$ remains uncovered.
		}
		\label{fig:Omega_l_uk}
	\end{center}
\end{figure}

	\begin{remark}
		$(u_k)$ is not a recovery sequence, 
		due to Theorem \ref{Thm:maintheorem}.
		It is proven in \cite{AcDa:94} that 
		$$
		\lim_{k \to +\infty}  
		\area(u_k,\Omega) =  \int_\Om|\mathcal M(\nabla \vortexmap)|~dx
		+2\pi l,
		$$
		and $2 \pi l$ has the meaning of the lateral area of the cylinder of 
		height $\longR$ and basis the unit disc. This surface is not a minimizer
		of the problem on the right-hand side of \eqref{43} 
		(where it corresponds to $h \equiv 1$).
	\end{remark}

	%%%%%%%%%%%%%%%%%%%%%%%%%%%%%%%%%%%%%%%%%%%%%%%%%%%%%%%%%%%%%%%%%%%%%%%
	\subsection{A non-optimal approximating sequence of maps: catenoid union a flap}
	\label{sec:catenoid_with_a_flap}
	%%%%%%%%%%%%%%%%%%%%%%%%%%%%%%%%%%%%%%%%%%%%%%%%%%%%%%%%%%%%%%%%%%%%%%%
	In this section we 
	discuss another example of a sequence $(u_k)$ converging to $u$. 
	We replace the cylinder lateral surface\footnote{In polar coordinates.} $[0,\longR]\times \{1\}\times(-\pi,\pi]$, which 
	contains the image of 
	$(r_k, l) \times (-\alpha_k,\alpha_k)$ through the map 
	$\Psi_k(x)=(|x|,u_k(x))$ 
	in the 
	example of Section 
	\ref{subsec:an_approximating_sequence_of_maps_with_degree_zero:cylinder}, with 
	half\footnote{
		For convenience, we consider the doubled segment $[0,2l]$,
		in order to define the catenoid; then we restrict
		the construction to $(0,\longR)$.} 
	of a catenoid union a flap (see Fig. \ref{fig:cat_flap}): calling this union $CF \res (0,\longR)\times \R^2$, 
	we have
	$$ CF=: \{(t, \overline\rho(t),\theta): t\in [0,2l],~ \theta \in (-\pi,\pi] \} \cup \{(t,r,0): t\in (0,2l),~r\in [\overline\rho(t),1]\},$$ 
	where $~\overline \rho(t):=a \cosh(\frac{t-l}{a}),$ and $a>0$ is such that $\overline\rho(0)=1$
	(and $\overline \rho(2l)=1$). 

	Notice that $CF$ ``spans'' 
	$\Big(\{0,2l\}\times \{1\}\times(-\pi,\pi]\Big) \cup \Big(  [0,2l]\times \{1\}\times\{0\} \Big)$, 
	which is the union of two unit circles joined by a segment. 
	
	Let $\reps>0, \theta_k>0$, $\overline \theta_k >\theta_k$ be
	such that   $\reps, \theta_k, (\overline \theta_k-\theta_k) \to 0^+$ 
	as $k\to +\infty$.
	Set $$\rho(t):= \overline \rho 
	\left(\frac{t-\reps}{l-\reps}l\right), \qquad t\in (\reps,l).$$
	We define $u_k := u$ in $\Omega \setminus \Big( \sourcedisk_{\reps}  \cup \{\alpha \in (-\overline \theta_k,\overline \theta_k)\}\Big)$,
	in particular
	$$
	u_k(\partial \sourcedisk_t \setminus \{\alpha\in (-\overline\theta_k, 
	\overline\theta_k) \})= \partial B_1 \setminus \{\theta \in (-\overline \theta_k, \overline \theta_k) \}, \qquad t \in (r_k,l).
	$$
	On $ \{\alpha \in (-\overline \theta_k,\overline \theta_k)\} \setminus  \sourcedisk_{\reps} $ we define $u_k$ in such a way that,
 for each $t\in (\reps,l)$, one has 
	\begin{align*}
		u_k\left(\partial \sourcedisk_t \cap 
		\{\pm\alpha\in (\theta_k, \overline{\theta_k}) \}\right)
		&=\partial B_1 \cap  \{\pm\theta \in (0, \overline \theta_k) \} ,\\
		u_k\left(\partial \sourcedisk_t \cap \{\pm\alpha\in (0,\theta_k) \}\right)
		&=\{(r,0)\in \overline {\rm B}_1: r\in [\rho(t),1]\} \cup\Big(\partial B_{\rho(t)} \cap  \{\theta\neq 0 \} \Big).
	\end{align*}
	See Fig. \ref{fig:D_k_in_Catenoid_EX} for a representation of the map $u_k$. The several parts of the image are run so that the winding number around the origin is always null.
	\begin{figure}
		\begin{center}
			\includegraphics[width=0.9\textwidth]{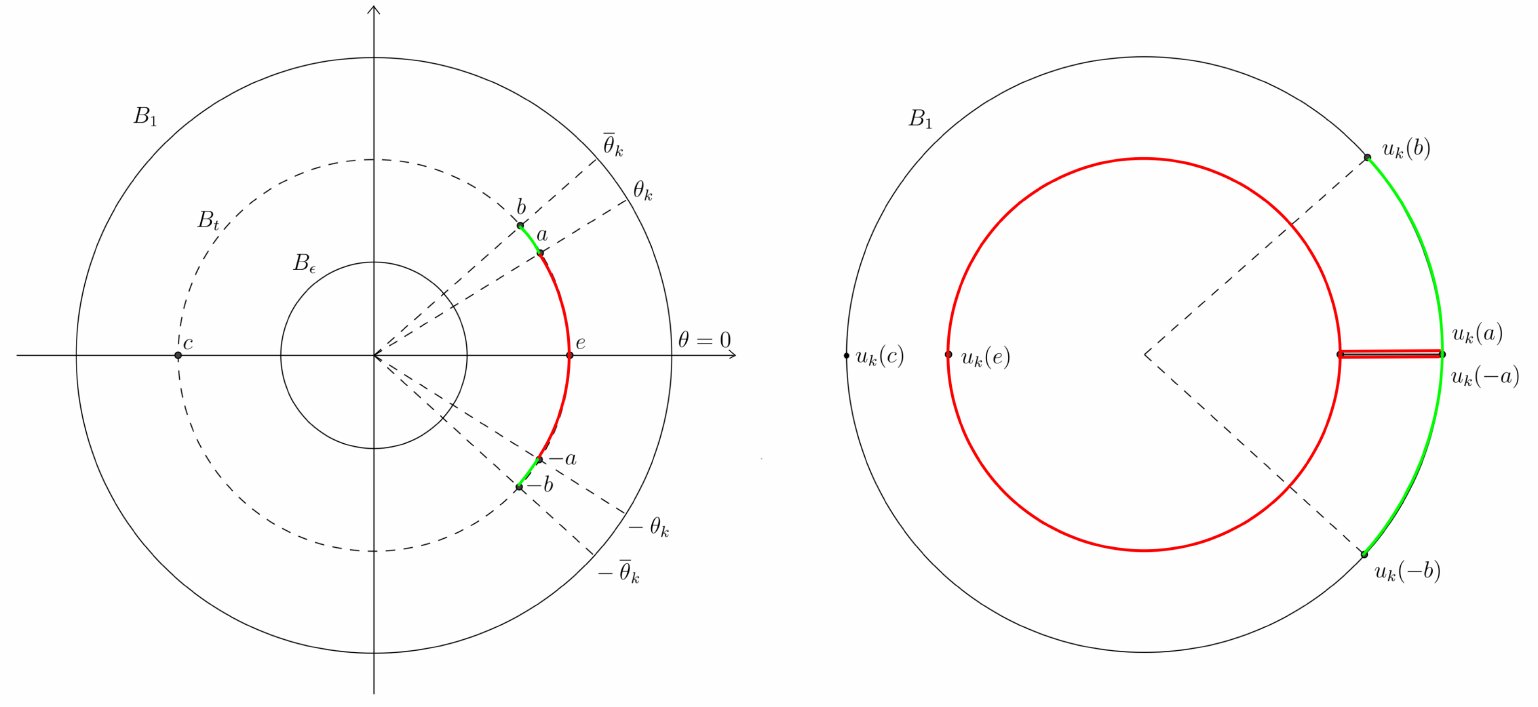}
			\caption{Source and target of the map $u_k$ in the example of Section \ref{sec:catenoid_with_a_flap}. The small interior circle
				in the right figure is a $t$-slice of a catenoid, whereas the horizontal segment is the $t$-section of the flap. The radius of the small circle is $\overline \rho(l)$.}
			\label{fig:D_k_in_Catenoid_EX}
		\end{center}
	\end{figure}
	To define $u_k$ on $\sourcedisk_\reps$ we adopt a construction similar to
	the one in \eqref{eqn:u_k_Cylinder}. First of all, $u_k(0,0):= (-1,0)$.
	Then,
	in 
	$\sourcedisk_\reps \cap \{\alpha\in (-\pi,\pi) \setminus(-\overline \theta_k, \overline \theta_k)\}$ we impose $u_k$ as in 
	\eqref{eqn:u_k_Cylinder} with $\overline \theta_k$ replacing $\alpha_k$. 
	In 
	$\sourcedisk_\reps \cap \{\alpha\in (-\overline \theta_k, \overline \theta_k)\}$ we require 
	$$u_k([0,r_k],\alpha):= \partial B_1 \cap \{\pm \theta \in ((u_k)_2(r_k,\alpha),\pi) \},
	\qquad \pm \alpha \in (0,\pi],
	$$
	where $(u_k)_2$ is the second (angular) coordinate of $u_k$.
 	\begin{figure}
		\begin{center}
			\includegraphics[width=0.4\textwidth]{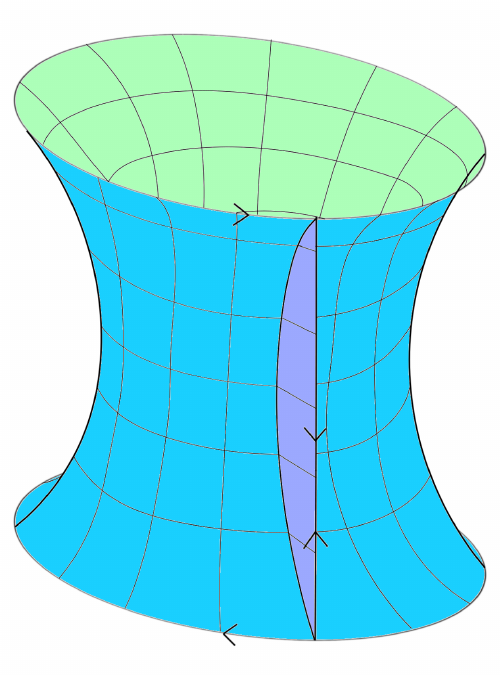}
			\caption{Catenoid union a flap, namely the set $CF$ (Section \ref{sec:catenoid_with_a_flap}). }
			\label{fig:cat_flap}
		\end{center}
	\end{figure}
	Hence 
	\begin{equation*}
		u_k(\partial \sourcedisk_t) \begin{cases}
			\subsetneq\partial B_1\qquad  &\text{if } t \in (0,\reps],\\
			= (\partial B_1) \cup \{(r,0)\in \overline B_1: r\in [\rho(t),1]\} \cup (\partial B_{\rho(t)})\qquad  &\text{if } t \in (\reps,l).
		\end{cases}
	\end{equation*}		
		
		\begin{remark}		Also in this case
			$(u_k)$ is not a recovery sequence, 
			due to Theorem \ref{Thm:maintheorem}, and the results in \cite{BES1,BES2}. For this particular sequence we have 
			$$
			\lim_{k \to +\infty}  
			\area(u_k,\Omega) =  \int_\Om|\mathcal M(\nabla \vortexmap)|~dx
			+\mathcal H^2
			({\rm catenoid})+2\, \mathcal H^2({\rm flap}).
			$$
			This surface is not a minimizer
			of problem on the right-hand side of \eqref{43}. However it is worth noticing that, by minimality property of the catenoid, it can be proved that the set $CF$, treated as an integral current, is $S_{\textrm{min}}$, the minimal vertical current 
closing the graph $\mathcal G_u$  of the vortex map in $\Om$ (see the discussion in the Introduction).
		\end{remark}
		
		%%%%%%%%%%%%%%%%%%%%%%%%%%%%%%%%%%%%%%%%%%%%%%%%%%%%%%%%%%%%%%%%%%%%%%%
		\subsection{The case of two discs}
		\label{sec:smoothing_by_convolution}
		%%%%%%%%%%%%%%%%%%%%%%%%%%%%%%%%%%%%%%%%%%%%%%%%%%%%%%%%%%%%%%%%%%%%%%%
		In \cite{AcDa:94}, the authors describe
		a sequence $(u_k)$ of maps converging to  
		the vortex map $\vortexmap$,  simply defined as follows:
		\begin{equation}\label{eqn:u_k_two_disks}
			\veps(r,\alpha):=\phi_k(r)u(r,\alpha),
		\end{equation}
		where $\phi_k:[0,\longR]\to [0,1]$ is a smooth function 
		such that $\phi_k=0$ in $[0,\frac{1}{k^2}]$, 
		$\phi_k=1$ in $[\frac{1}{k},l ]$, 
		and $0\leq \phi_k^\prime\leq 2k$. 
In this case
			$(u_k)$ is a recovery sequence for $\longR$ 
			sufficiently large, due to \cite[Lemma 4.2]{AcDa:94}.
			We have 
			$$
			\lim_{k \to +\infty}  
			\area(u_k,\Omega) =  \int_\Om|\mathcal M(\nabla \vortexmap)|~dx
			+\pi,
			$$
			and $\pi$ has the meaning of the area of the unit disc.
			This surface, for $\longR$ sufficiently large, is a minimizer
			of problem on the right-hand side of \eqref{43} 
			(where it corresponds to $h \equiv -1$).

\section{Proof of Theorem \ref{Thm:maintheorem}}\label{sec:final}
In this section we prove Theorem \ref{Thm:maintheorem}.
 To this aim, we need to construct a 
sequence $(\veps) \subset \textrm{Lip}(\BallR, \R^2)$ 
converging to $\vmap$ in $L^1(\BallR,\R^2)$ such that
\begin{equation*}
\lim_{k \to +\infty}\area(\veps, \BallR)\leq
\int_{\BallR}\vert \M(\nabla \vmap)\vert dx + \FB(\hstar,\psi^\star),
\end{equation*}
where $(\hstar, \psi^\star)$ is a pair
minimizing $\FB$ as in Theorem \ref{Thm:existenceofminimizer}. 
We may assume that $h^\star$ is not identically
$-1$, otherwise the result follows from \cite{AcDa:94} (and a recovery sequence is provided as in \eqref{eqn:u_k_two_disks}).

We will specify various subsets of $\BallR$ and define the sequence 
$(\veps)$ on each of these sets (see Fig. \ref{fig:picture_alaa_last_section}). More precisely, we will define $u_k$ 
as a map taking values in $\mathbb S^1$ in the largest sector (step 1). This construction is similar 
to the one in \cite{AcDa:94} (see also Remark \ref{rem:below} below). The contribution of the area in this sector will equal, as $k\rightarrow
+\infty$, the first term in \eqref{eq:upper_bound_recovery}. The second term 
will be instead provided by the contribution of $u_k$ in region $C_k\setminus B_{r_k}$ (step 2), where we will need the aid of 
the functions $(\hstar, \psi^\star)$  (suitably regularized, in order to render $u_k$ Lipschitz continuous). The other regions surrounding $C_k\setminus B_{r_k}$ are needed to glue $u_k$ between the aforementioned regions. This is done in steps 3, 4 and 5, 
where it is also proven that the corresponding 
area contribution is negligible. Finally, in steps 6 and 7 
we show the crucial estimates to prove \eqref{eq:upper_bound_recovery}. 
In Fig. \ref{fig:picture_alaa_last_section} 
this subdivion of the domain $\Omega$ is drawn.

\begin{Remark}\label{rem:below}\rm
Our construction 
differs from the one in \cite{AcDa:94}, even when in place of 
$(\hstar,\psi^\star)$ we use $(1, \sqrt{1-s^2})$ 
(\textit{i.e.},
the one in Section \ref{subsec:an_approximating_sequence_of_maps_with_degree_zero:cylinder})
in the following sense.
We use the full 
graph of $\pm\psi^\star$ to construct 
$u_k$ and therefore, in the case  
when $(\hstar,\psi^\star)$ is replaced by $(1, \sqrt{1-s^2})$, 
the image of $u_k$ covers the whole cylinder and not
only a part of it. Since $h^\star$ {\it may be
not identically $1$} (and actually is not explicit in general), 
the presence of a new set $T_k$
is now needed, as an intermediate region to glue the 
trace of $u_k$ along the two segments $\{\alpha=  \pm \overline\theta_k\}$.
The image set $u_k(T_k)$ covers a small
part of the unit circle.
See Fig. \ref{fig:picture_alaa_last_section},
where $T_k$ is represented as the union of the two 
thin sectors in $\Om$. To glue all the piecese in order that $u_k$ is Lipschitz, it will be useful to have two transition regions, one in a ball $B_{r_k/2}$ and one in the annulus $B_{r_k}\setminus B_{r_k/2}$.
It is worth noticing that the curve $u_k\res\partial {\rm B}_t$ has null winding number around the origin, for all $t\in(0,l)$. 
\end{Remark}
Let $k \in \mathbb N$ and let $(\reps),(\thetaeps), (\thetaepsbar)$ 
be infinitesimal sequences of positive numbers
such that $\thetaepsbar-\thetaeps=:\deltaeps>0$. 
We suppose\footnote{This assumption is used only in step 7.} 
\begin{equation}\label{eq:k_theta_k_to_zero}
\lim_{k \to +\infty}(\theta_k k)=0.
\end{equation}
Let  $\Balleps$ be the 
open disc centered at the origin with radius $\reps$, and 
\begin{equation}\label{eq:C_k}
\Coneeps:=\{(r,\alpha)\in 
[0,\longR) \times [0,2\pi):
\alpha \in [0,\thetaeps]\cup[2\pi-\theta_k,2\pi)\},
\end{equation}
be the half-cone in $\Omega$, 
with vertex at the origin and aperture equal to $2\thetaeps$, see
Fig. \ref{fig:picture_alaa_last_section}. 
We set 
$$
C_k^+:=C_k\cap\{\alpha\in[0,\theta_k]\}, \quad
C_k^-:=C_k\cap\{\alpha\in[2\pi-\theta_k,2\pi]\},
$$
and we divide $\Coneeps\cap (\BallR\setminus\Balleps)$ into two sets
\begin{equation}\label{eq:C_k_plus_minus}
\begin{aligned}
C_k\setminus \Balleps:=
\big(C_k^+ \setminus\Balleps\big)\cup\big(
C_k^- \setminus\Balleps\big).
\end{aligned}
\end{equation}
\begin{figure}
\begin{center}
    \includegraphics[width=0.9\textwidth]{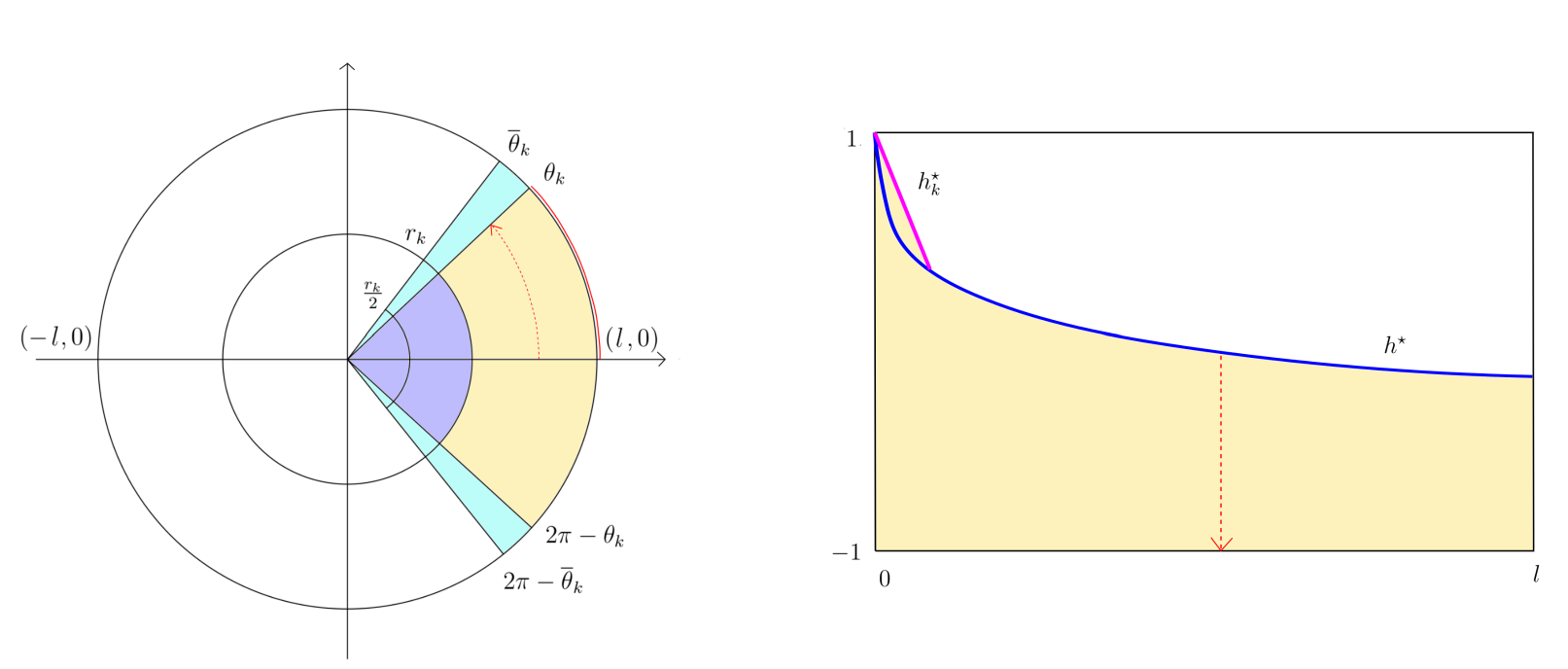}
\caption{
On the left the subdivision of $\sourcedisk_l$ in sectors. 
Specifically, the sectors 
$C_k^+\setminus \Balleps$ 
and $C_k^-\setminus \Balleps$ are emphasized in 
light grey. The map $\mathcal T_k$ defined
	in \eqref{eqn:changeofvariableT} sends $C_k^+\setminus \Balleps$ 
 in the (reflected) subgraph of $h^\star_k$ in $R_\longR$, 
depicted on the right; it maps the segment joining $(r_k,0)$ to $(1,0)$ onto the graph of $h^\star_k$, and the radius corresponding to $\alpha=\theta_k$ onto the basis of $R_l$, following the orientation emphasized by the dashed arrow. The graph of $h^\star_k$ starts 
linearly from the point $(0,1)$ in the interval $[0,1/k)$ with negative derivative,
then joins
(and next coincides) 
 with the graph of $h^\star$.  The definition of  $u_k$ in $C_k^+\setminus \Balleps$
 makes use of this parametrization of $SG_{h^\star_k} \cap 
\overline R_\longR$ (see \eqref{vepsonCeps}). This parametrization needs a reflection,
in order to glue $u_k$ on the 
horizontal segment $\{\alpha=0\}$ with the definition of $u_k$ 
in $C_k^-\setminus \Balleps$. 
}
\label{fig:picture_alaa_last_section}
\end{center}
\end{figure}
Finally,
let  
\begin{equation}\label{eq:T_k}
\Teps:=\{(r,\alpha)\in [0,\longR) \times [0,2\pi):\alpha \in [\thetaeps,\thetaepsbar]\cup[2\pi-\thetaepsbar,2\pi-\thetaeps]\}.
\end{equation} 
{\it Step 1.} 
Definition of $\veps$ on $\overline{\BallR \setminus ( \Coneeps \cup \Teps)}$.

In this step our 
construction is similar to the one in \cite[Lem. 5.3]{AcDa:94},
see also \eqref{eqn:u_k_Cylinder}; in order to define $u_k$, 
in the source we use polar  
coordinates $(r,\alpha)$ and Cartesian coordinates in the target.
 Define
\begin{equation}\label{eqn:vepsoutofthecone}
\veps(r,\alpha):=
\begin{cases}
      \vmap(r,\alpha) = (\cos \alpha, \sin \alpha), & 
r \in (r_k/2, \longR), \alpha \in [\overline \theta_k, 2\pi - 
\overline \theta_k],
\\
     \Big(\cos (\frac{2r}{\reps}(\alpha -\pi)+\pi),
\sin (\frac{2r}{\reps}(\alpha -\pi)+\pi)\Big), & 
r \in [0,r_k/2], \alpha \in [\overline \theta_k, 2\pi - \overline 
\theta_k].
    \end{cases} 
\end{equation}
Obviously
\begin{align}
&\veps(0,0)=(-1,0) = u_k(r,\pi), \qquad r \in [0,\longR),
\nonumber
\\
&\veps(r,\thetaepsbar)=(\cos \thetaepsbar, \sin  \thetaepsbar), \qquad 
\veps(r,2\pi-\thetaepsbar)=(\cos \thetaepsbar, \sin  (-\thetaepsbar)),
\qquad r \in (\reps/2,l),
\nonumber
\\
&\veps(r,\thetaepsbar)=\Big(
\cos(\frac{2r}{\reps}(\thetaepsbar -\pi)+\pi)~,~ 
\sin (\frac{2r}{\reps}(\thetaepsbar -\pi)+\pi)\Big),
\qquad r \in [0,\reps/2],
\nonumber
\\
&\veps(r,2\pi-\thetaepsbar)=\Big(
\cos(\frac{2r}{\reps}(\pi-\thetaepsbar )+\pi)~,~ 
\sin (\frac{2r}{\reps}(\pi-\thetaepsbar)+\pi)\Big),\qquad 
r \in [0,\reps/2].
\label{eqn:vepsboundaryinBeps}
\end{align}

\smallskip
The relevant contribution to the area of the graph of $u_k$ is the one
in region $C_k$, and more specifically in $C_k \setminus B_{r_k}$;
it is in this region that we need to use a minimizing pair of $\FB$.

{\it Step 2.} 
Definition of $\veps$ on $\Coneeps\setminus \Balleps$.

We first need a regularization of $h^\star$:  assuming without 
loss of generality $1/k < \longR$,
we define 
\begin{equation}\label{def_h_k}
h^\star_k (w_1):=
\begin{cases}
h^\star(w_1) & {\rm for~  }w_1\in [\frac1k,\longR],
\\
k\left(h^\star (\frac1k)- \hstar(0)\right)w_1+\hstar(0)& {\rm for~  }w_1\in [0,\frac1k),
\end{cases}
\end{equation}
where we recall that $h^\star(0)=1$ (see Theorem \ref{Thm:existenceofminimizer}), 
and we set $h^\star_k(w_1):=h^\star_k(2\longR -w_1)$ for $w_1\in 
[\longR,2\longR]$ (see Fig. \ref{fig:picture_alaa_last_section}, right). Notice that $\hstark(0)=1$, 
$\h^\star_k \in \Lip([0,2\longR])$ and
the convexity of $h^\star$ implies that also $h^\star_k$ is convex,
$h^\star_k\geq h^\star$, and therefore by Lemma \ref{lem:properties_of_psi_m} (i) we see that
$(h^\star_k, \psi^\star_k) 
\in X_{2\longR}^{{\rm conv}}$,
where 
$\psi^\star_k$ is the approximation of $\psi^\star$ considered in Lemma 
\ref{lem:properties_of_psi_m} (with $k=m$), see formula
\eqref{def_psi_m}.
Again by Lemma
\ref{lem:properties_of_psi_m}, 
$\FB(h^\star_k,\psi^\star_k)=
\FB(h^\star,\psi^\star_k)+ \int_0^{2\longR} 
\left( h^\star_k(w_1) -\h^\star(w_1) \right)~dw_1\rightarrow \FB(h^\star,\psi^\star)$ as $k\rightarrow +\infty$. 

We start with the construction of $u_k$ on $C_k^+\setminus \Balleps$. Set
\begin{align}
&\teps:[\reps,\longR]\to[0,\longR],&\qquad &\teps(r):=\frac{\longR}{\longR-\reps}(r-\reps), \label{eqn:teps}\\
&\seps:[\reps,\longR]\times[0,\thetaeps]\to [-1,1],&\qquad &\seps(r,\alpha):=\frac{1 + \hstark(\teps(r))}{\thetaeps}\alpha -\hstark(\teps(r)).\label{eqn:seps}
\end{align}
Note that $\seps(r,\cdot):[0,\thetaeps]\to [-\hstark(\teps(r)),1]$ is a 
bijective increasing function, for any $r \in [r_k, \longR]$, 
and 
\begin{align}
&\seps(r,0)=-\hstark(\teps(r)) \quad{\rm for~any}~ 
r \in [r_k,\longR],  \text{ in particular }\seps(\reps,0)=-1,
\label{eqn:sepsboundarythetazero}
\\
&\seps(r,\thetaeps)=1, \qquad r\in[\reps,l]
\label{eqn:sepsboundarythetaeps},
\\
& \seps(r_k,\alpha)=\frac{2\alpha}{\theta_k} -1, \qquad \alpha\in[0,\theta_k].
\label{eq:s_k_r_k_theta}
\end{align}
We have, for all $r\in[\reps,l]$ and $\alpha\in[0,\theta_k]$,
\begin{align}
&
\teps'(r)
=\frac{\longR}{\longR-\reps},\label{eqn:tepsderiv}\\
&\partial_\alpha \seps(r,\alpha)=\frac{1 +\hstark(\teps(r))}{\thetaeps},
\label{eqn:sepsthetaderiv}
\end{align}
and, for almost every $r\in [\reps,l]$ and
all $\alpha\in[0,\theta_k]$, 
\begin{align}
\partial_r \seps(r,\alpha) =\left(\frac{\alpha}{\thetaeps}-1\right)
\teps'(r)  \hstark'(\teps(r)) = 
\frac{\longR}{\longR-\reps}
\left(\frac{\alpha}{\thetaeps}-1\right)  
\hstark'(\teps(r)).
\label{eqn:sepsrderiv}
\end{align}
Moreover  we define
\begin{equation}\label{eqn:rfn}
\invr : [0,\longR]\to[\reps,\longR], \qquad
\invr(\tcoord) :=\frac{\longR-\reps}{\longR}\axialcoordofcylinder+\reps ~
\end{equation}
to be the inverse of $\tau_k$ and, recalling that 
$\overline R_l = [0,\longR] \times [-1,1]$, 
\begin{equation}
\label{eqn:thetafn} 
\invtheta
:\Omegahstark \cap \overline R_l  \to [0,\thetaeps],
\qquad
\invtheta(\axialcoordofcylinder,\scoord) :=\frac{\thetaeps}{1 + \hstark (\axialcoordofcylinder)}
(\hstark(\axialcoordofcylinder)-\scoord).
\end{equation}
Notice that 
$\invtheta(\axialcoordofcylinder,\cdot):
[-1,\hstark(\axialcoordofcylinder)]
\to [0,\thetaeps]$ is a linearly decreasing bijective function\footnote{We recall that in our hypothesis $h_k^* > -1$  by Theorem \ref{Thm:existenceofminimizer} (i).}
	
The map 
\begin{equation}\label{eqn:changeofvariableT}
\mathcal T_k:C_k^+\setminus \Balleps
\to 
\Omegahstark
\cap 
\overline R_l, 
\qquad \mathcal T_k(r,\alpha):=(\teps(r),-\seps(r,\alpha)),
\end{equation} 
is invertible, and its inverse is the map
\begin{equation}\label{eqn:changeofvariableTinverse}
\mathcal T_k^{-1}:
\Omegahstark\cap 
\overline R_l
\to C_k^+\setminus \Balleps, \qquad 
\mathcal T_k^{-1}(\axialcoordofcylinder,\scoord)
:=(\invr(\axialcoordofcylinder),\invtheta(\axialcoordofcylinder,\scoord)).
\end{equation} 
The modulus of the determinant of the Jacobian of $\mathcal T_k^{-1}$ is given by 
\begin{equation}\label{eqn:changeofvariabledeterminant}
\vert J_{\mathcal T_k^{-1}}\vert=\left( 
\frac{\longR -\reps}{\longR}\right)
\frac{\thetaeps}{1 +\hstark(\axialcoordofcylinder)}.
\end{equation}
We set 
\begin{equation}\label{vepsonCeps}
\veps(r,\alpha)
:=\Big(\seps(r,\alpha),\psionestar_k
\big(\mathcal T_k(r,\alpha)\big)\Big) =
\Big(u_{k1}(r,\alpha),
u_{k2}(r,\alpha)\Big), \qquad 
r \in [r_k, l], \alpha \in [0,\theta_k].
\end{equation}
Observe that, using the definition of $\psi_k^\star$,
\begin{equation}\label{eq_13.28} 
\begin{aligned}
&\veps \in {\rm Lip}(C_k^+\setminus \Balleps, \R^2),
\\
&\veps(r,\thetaeps)=(\seps(r,\thetaeps), 
\psi_k^\star(\mathcal T_k(r, \theta_k)))=(1,0),
\\
&\veps(r,0)=(-\hstark(\teps(r)), \psionestar_k(\teps(r),\hstark(\teps(r))))
=(-\hstark(\teps(r)),0),
\\
&\veps(\reps,\alpha)
=(\seps(\reps,\alpha),\psionestar_k(0,-\seps(\reps,\alpha)))
=(\seps(\reps,\alpha),\varphi_k(0,-\seps(\reps,\alpha))),
\end{aligned}
\end{equation}
for $r \in [r_k, l]$ and $\alpha \in [0, \theta_k]$, 
as it follows from
\eqref{eqn:teps}, 
\eqref{eqn:sepsboundarythetazero},  \eqref{eqn:sepsboundarythetaeps}, and \eqref{eqn:psionestarboundary},
where $\varphi_k$ is defined in \eqref{def_varphi_m} (with $k=m$).

Eventually we define $u_k$ on $C_k^-\setminus \Balleps$ as
\begin{align}\label{vepsonCeps-}
\veps(r,\alpha)
:=
(u_{k1}(r,2\pi-\alpha),
-u_{k2}(r,2\pi-\alpha)),
\qquad
r \in [r_k, l], \alpha \in [2\pi-\theta_k, 2\pi). 
\end{align}
It turns out
\begin{align*}
&\veps  \in {\rm Lip}(C_k^-\setminus \Balleps, \R^2),
\\
&\veps(r,2\pi-\thetaeps)=(1,0),
\\
&\veps(r,2\pi)=(-\hstark(\teps(r)), -\psionestar_k(\teps(r),\hstark(\teps(r))))
=(-\hstark(\teps(r)),0),
\\
&\veps(\reps,\alpha)=(\seps(\reps,2\pi-\alpha), -\psionestar_k(0,-\seps(\reps,2\pi-\alpha))),
\end{align*}
for $r\in[r_k,l]$, $\alpha\in[2\pi-\theta_k,2\pi)$.

The area of the graph of $u_k$ on $C_k \setminus B_{r_k}$
will be computed in step 7. 
\smallskip

{\it Step 3.} 
Definition of $\veps$ on $\Coneeps\cap(\overline B_{r_k}\setminus {\rm B}_{r_k/2})$
and its area contribution.  

Let $G_{\psi^\star_k(0,\cdot)}\subset\R^2$ (resp.
$G_{\psi^\star(0,\cdot)}\subset\R^2$)
denote the graph of $\psi^\star_k(0,\cdot)$ (resp. 
of $\psi^\star(0,\cdot)$)
on $[-1,1]$.
We 
introduce the retraction map
$\Upsilon:(\R\times [0,+\infty))\setminus O
\subset \R^2_{{\rm target}}
\rightarrow  G_{\psi^\star(0,\cdot)}\subset \R^2_{{\rm target}}$, $O=(0,0)$, defined by
\begin{align*}
\Upsilon(p)=q:=G_{\psi^\star(0,\cdot)}\cap \ell_{Op}\qquad \forall p\in (\R\times [0,+\infty))\setminus O,
\end{align*}
where $\ell_{Op}$ is the line passing through 
$O$ 
 and $p$. Then  $\Upsilon$ is well-defined and it is 
Lipschitz continuous in a neighbourhood of $G_{\psi^\star(0,\cdot)}$ in $\R\times [0,+\infty)$. 
We also define 
$$
\Upsilon_k:G_{\psi^\star_k(0,\cdot)}\rightarrow G_{\psi^\star(0,\cdot)}
$$
as the 
restriction of $\Upsilon$ to $G_{\psi^\star_k(0,\cdot)}$;
see Fig. \ref{fig:graphof_h}.
 As a consequence, since for $k \in \NN$ large enough $G_{\psi_k^\star(0,\cdot)}$ 
is contained in a neighbourhood of $G_{\psi^\star(0,\cdot)}$, 
we have that $\Upsilon_k$ is Lipschitz continuous with Lipschitz constant 
independent of $k$.
Notice also that $\Upsilon_k((-1,0))=(-1,0)$ and $\Upsilon_k((1,0))=(1,0)$. 

We define $u_k$ on $C_k^+\cap (\overline {\rm B}_{r_k}\setminus {\rm B}_{r_k/2})$ setting, 
for $r\in[\frac{r_k}{2},r_k]$ and $\alpha\in[0,\theta_k]$,
\begin{align*}
u_k(r,\alpha):=\Big(2-\frac{2r}{r_k}\Big)
\Upsilon_k\big(s_k(r_k,\alpha),\psi_k^\star(0,-s_k(r_k,\alpha))\big)+
\Big(\frac{2r}{r_k}-1\Big)\big(s_k(r_k,\alpha),\psi_k^\star(0,-s_k(r_k,\alpha))\big).
\end{align*}
We have
$$
u_k(r_k,\alpha)=(s_k(r_k,\alpha),\psi_k^\star(0,-s_k(r_k,\alpha))),
$$
so that $u_k$ glues, on $C_k^+ \cap \partial {\rm B}_{r_k}$, 
 with the values obtained in step 2 (last formula in \eqref{eq_13.28}), and
$$
u_k(r,\theta_k)=(1,0),\qquad u_k(r,0)=(-1,0).
$$
This formula shows that $u_k$ also glues, 
on $C_k^+ \cap \{(r,\alpha): r \in [r_k/2,r_k], \alpha \in \{0,\theta_k\}\}$, 
with the values obtained in step 2 (second and third formula in \eqref{eq_13.28}). Moreover
\begin{equation}\label{eq:step_3_r_k_2}
	u_k(r_k/2,\alpha)=\Upsilon_k\big(s_k(r_k,\alpha),\psi_k^\star(0,-s_k(r_k,\alpha))\big),
\qquad \alpha \in [0,\theta_k].
\end{equation}
In addition, using \eqref{eq:s_k_r_k_theta},
the derivatives of $u_k$ satisfy, 
for $r\in(\frac{r_k}{2},r_k)$ and $\alpha\in(0,\theta_k)$, 
\begin{align*}
&\partial_ru_k(r,\alpha)=-\frac{2}{r_k}\Upsilon_k\big(s_k(r_k,\alpha),\psi_k^\star(0,-s_k(r_k,\alpha))\big)+\frac{2}{r_k}\big(s_k(r_k,\alpha),\psi_k^\star(0,-s_k(r_k,\alpha))\big),\\
&\partial_\alpha u_k(r,\alpha)=
\Big(2-\frac{2r}{r_k}\Big)\nabla \Upsilon_k\big(s_k(r_k,\alpha),\psi_k^\star(0,-s_k(r_k,\alpha))\big)
\cdot\Big(\frac{2}{\theta_k},-\frac{2}{\theta_k}\partial_{w_2}\psi_k^\star(0,-s_k(r_k,\alpha))\Big)\nonumber\\
&\qquad\qquad\qquad +\Big(\frac{2r}{r_k}-1\Big)\Big(\frac{2}{\theta_k},-\frac{2}{\theta_k}\partial_{w_2}
\psi_k^\star(0,-s_k(r_k,\alpha))\Big),
\end{align*}
so that 
\begin{align*}
&|\partial_ru_k(r,\alpha)|\leq \frac{4}{r_k},\\
&|\partial_\alpha u_k(r,\alpha)|\leq 
\frac{2(\widehat C+1)}{\theta_k}(|\partial_{w_2}
\psi_k^\star(0,-s_k(r_k,\alpha))|+1),
\end{align*}
where $\widehat C$ is a positive constant independent of $k$, which bounds the gradient of $\Upsilon_k$.
Since $\psi_k^\star$ 
is Lipschitz, we deduce that $u_k$ is Lipschitz 
continuous\footnote{The Lipschitz constant of $u_k$ on this set turns out to be unbounded with respect to $k$.} 
on $C_k^+\cap(\Balleps\setminus {\rm B}_{r_k/2})$.

Furthermore the image of $(\frac{r_k}{2},r_k)\times (0,\theta_k)$ through the map $(r,\alpha)\mapsto u_k(r,\alpha)$
is  the 
region enclosed by $G_{\psi^\star_k}$ and $G_{\psi^\star}$ (with multiplicity $1$). 
The area of this region is infinitesimal as $k\rightarrow +\infty$, 
so that, by the area formula,
$$\int_{r_k/2}^{r_k}\int_0^{\theta_k}r|Ju_k(r,\alpha)|d\alpha dr=o(1)
\qquad
{\rm as~} k\rightarrow +\infty.
$$
Hence, using the fact that the gradient in polar
coordinates is $(\partial_r,\frac1r\partial_\alpha)$, 
we eventually estimate 
\begin{align}\label{estimate_area_step3}
\int_{r_k/2}^{r_k}\int_0^{\theta_k}r|\mathcal M(\nabla u_k)|d\alpha 
dr&\leq 
\int_{r_k/2}^{r_k}\int_0^{\theta_k}\Big(r+\frac{4r}{r_k}+\frac{C}{\theta_k}|\partial_{w_2}
\psi_k^\star(0,1-\frac{2\alpha}{\theta_k})|+\frac{C}{\theta_k}\Big)~d\alpha dr+o(1),
\nonumber
\\
&=o(1)+C\frac{r_k}{2\theta_k}\int_0^{\theta_k}|\partial_{w_2}\psi_k^\star(0,1-\frac{2\alpha}{\theta_k})|d\alpha=o(1)
\end{align}
as $k\rightarrow +\infty$.
In the last equality we use that $|\partial_{w_2}\psi^\star_k(0,\cdot)|
\leq|\partial_{w_2}\psi^\star(0,\cdot)|$, which is integrable via the change of variables $w_2=1-\frac{2\alpha}{\theta_k}$ (it also makes $\theta_k$ disappear at the denominator in front of the integral 
in \eqref{estimate_area_step3}).

This proves that 
the contribution of area of the graph of 
$u_k$ over $C_k^+\cap(\Balleps\setminus {\rm B}_{r_k/2})$ is infinitesimal as $k\rightarrow +\infty$.

\begin{figure}
	\centering
	\includegraphics[scale=0.6]{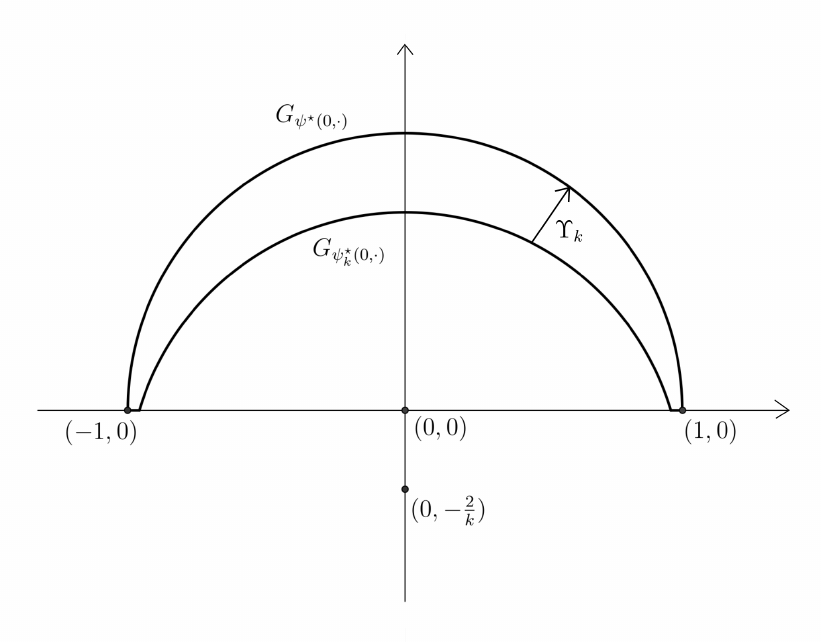}
	\caption{the graphs of the functions $\psi^\star_k(0,\cdot)$ 
and $\psi^\star(0,\cdot)$; these contain arcs of circle centered at $(0,0)$ and $(0,-\frac2k)$ respectively. The map $\Upsilon_k$ is emphasized,
and turns out to be the restriction of $x\mapsto \frac{x}{|x|}$ on $\psi^\star_k(0,\cdot)$.}
	\label{fig:graphof_h}
\end{figure}

Eventually, for $r\in[r_k/2,r_k]$, $\alpha\in[2\pi-\theta_k,2\pi)$, we set
\begin{align}
u_k(r,\alpha):=(u_{k1}(r,2\pi-\alpha),-u_{k2}(r,2\pi-\alpha)).
\end{align}
Observe that, thanks to \eqref{vepsonCeps-}, $u_k$ is continuous on $\partial {\rm B}_{r_k}$, and 
similar estimates as in \eqref{estimate_area_step3} hold on $(\sourcedisk_{r_k}
\setminus \sourcedisk_{r_k/2})\cap C_k^-$.
\smallskip

{\it Step 4.} Definition of $\veps$ on $\Coneeps\cap \sourcedisk_{r_k/2}$
and its area contribution.

We start with the construction of $u_k$ on $C_k^+\cap {\rm B}_{r_k/2}$. 
For $r\in[0,r_k/2]$  and $\alpha\in [0,\theta_k]$ we set
\begin{align}\label{u_k_step4}
u_k(r,\alpha):=\Upsilon_k\Big(\frac{4r\alpha}{r_k\theta_k}-1,\psi_k^\star
\big(0,1-\frac{4r\alpha}{r_k\theta_k}\big)\Big).
\end{align}
First we observe that
$$
u_k\Big(\frac{r_k}{2},\alpha\Big)=\Upsilon_k\Big(\frac{2\alpha}{\theta_k}-1,\psi_k^\star
\big(0,1-\frac{2\alpha}{\theta_k}\big)\Big), \qquad \alpha \in (0,\theta_k),
$$
so that  $u_k$ is continuous on $C_k^+ \cap
\partial {\rm B}_{r_k/2}$ (see \eqref{eq:step_3_r_k_2} and \eqref{eq:s_k_r_k_theta}), and  
\begin{align}\label{eq:u_k_lip_cont}
&u_k(r,\theta_k)=\Upsilon_k\Big(\frac{4r}{r_k}-1,\psi_k^\star(0,1-\frac{4r}{r_k})\Big),
\\
&u_k(r,0)=(-1,\psi_k^\star(0,1))=(-1,0).
\end{align}
Direct computations lead to the following estimates:
\begin{align}
&|\partial_ru_k(r,\alpha)|\leq \widehat C\frac{4\alpha}{r_k\theta_k}\Big(1+|\partial_{w_2}\psi^\star_k(0,1-\frac{4\alpha r}{r_k\theta_k})|\Big),\\
&|\partial_\alpha u_k(r,\alpha)|\leq \widehat C\frac{4r}{r_k\theta_k}\Big(1+|\partial_{w_2}\psi^\star_k(0,1-\frac{4\alpha r}{r_k\theta_k})|\Big),
\end{align}
where $\widehat C$ is the constant bounding the gradient of $\Upsilon_k$ as in 
step 3.
Finally, since by \eqref{u_k_step4} $u_k$ takes values in $\mathbb S^1\subset\R^2$, we have
$Ju_k(r,\alpha)=0$ for all $r\in (0,r_k/2)$, $\alpha\in[0,\theta_k]$. 
Hence, the area of the graph of $u_k$ on $C_k^+\cap {\rm B}_{r_k/2}$ is 
\begin{align*}
\int_0^{r_k/2}\int_0^{\theta_k}r|\mathcal M(\nabla u_k)(r,\alpha)|~d\alpha 
dr\leq \int_0^{r_k/2}\int_0^{\theta_k} (r+C) 
+\frac{C}{\theta_k}+\frac{Cr}{r_k}(1+\frac{1}{\theta_k})|\partial_{w_2}\psi^\star_k(0,1-\frac{4\alpha r}{r_k\theta_k})|d\alpha dr,
\end{align*}
where 
$C$ is a positive constant independent of $k$. Exploiting 
that $|\partial_{w_2}\psi^\star_k(0,\cdot)|\leq|\partial_{w_2}\psi^\star(0,\cdot)|$, we can 
estimate the right-hand side of the previous formula as follows:
\begin{equation}
\label{estimate_area_step4}
\begin{aligned}
&C\int_0^{r_k/2}\int_0^{\theta_k}
\frac{r}{r_k}
\Big(1+\frac{1}{\theta_k}\Big)|\partial_{w_2}\psi^\star(0,1-\frac{4\alpha r}{r_k\theta_k})|d\alpha dr+o(1)
\\
\leq&C\int_0^{r_k/2}\int_{-1}^1\theta_k
\Big(1+\frac{1}{\theta_k}\Big)|\partial_{w_2}\psi^\star(0,w_2)|dw_2dr+o(1)\\
\leq & C\int_0^{r_k/2}(\theta_k+1)dr+o(1)=o(1),
\end{aligned}
\end{equation}
where $o(1)\rightarrow0$ as $k\rightarrow +\infty$, and $C$ is a positive constant independent of $k$ which might change from line to line.

In $C_k^-\cap \sourcedisk_{r_k/2}$ we set, for $r\in [0,r_k/2)$, $\alpha\in[2\pi-\theta_k,2\pi)$, 
\begin{align*}
u_k(r,\alpha):=(u_{k1}(r,2\pi-\alpha),-u_{k2}(r,2\pi-\alpha)).
\end{align*}
 
Similar estimates as in \eqref{estimate_area_step4} 
for the area of the graph of $u_k$ hold on $C_k^- \cap \sourcedisk_{r_k/2}$.

\medskip

{\it Step 5.} Definition of $\veps$ on $\Teps$
and its area contribution.

We first construct $u_k$ on $\Teps \cap \{(r,\alpha): r\in [0,r_k/2],\alpha \in [\thetaeps, \thetaepsbar]\}$.
We define $\beta_k: [0,r_k/2]\times [\theta_k,\overline\theta_k]\rightarrow [0,\pi]$ as
\begin{align*}
\beta_k(r,\alpha):=\frac{\overline\theta_k-\alpha}{\overline\theta_k-\theta_k}\alpha_k(r)+(1-\frac{\overline\theta_k-\alpha}{\overline\theta_k-\theta_k})\big(\frac{2r}{r_k}(\overline\theta_k-\pi)+\pi\big),
\end{align*}
where 
$$\alpha_k(r):=\arccos\big(\Upsilon_{k1}(\frac{4r}{r_k}-1,\psi^\star_k(0,1-\frac{4r}{r_k}))\big), \qquad
r \in [0,r_k/2].
$$
Notice that $\alpha_k$ is decreasing and takes values in 
$[0,\pi]$.
Therefore we set
\begin{equation}\label{eq:u_k_conetto}
u_k(r,\alpha):=\big(\cos(\beta_k(r,\alpha)),\sin (\beta_k(r,\alpha))\big),
\qquad
(r,\alpha)\in [0,r_k/2]\times [\theta_k,\overline\theta_k].
\end{equation}
One checks that $\beta_k(r,\theta_k)= \alpha_k(r)$,
$\beta_k(r,\overline \theta_k) = \frac{2r}{r_k}(\overline \theta_k-\pi)+\pi$
(see also \eqref{eqn:vepsoutofthecone}), and   
\begin{align*}
&\alpha_k(r_k/2)=0,\\
&u_k(r_k/2,\alpha)=\big(\cos((1-\frac{\overline\theta_k-\alpha}{\overline\theta_k-\theta_k})\overline\theta_k),\sin((1-\frac{\overline\theta_k-\alpha}{\overline\theta_k-\theta_k})\overline\theta_k)\big),\\
&u_k(r,\theta_k)=\big(\cos(\alpha_k(r)),\sin (\alpha_k(r))\big)=\Upsilon_k(\frac{4r}{r_k}-1,\psi^\star_k(0,1-\frac{4r}{r_k})),\\
&u_k(r,\overline\theta_k)=\big(\cos(\frac{2r}{r_k}(\overline\theta_k-\pi)+\pi),\sin(\frac{2r}{r_k}(\overline\theta_k-\pi)+\pi)\big),
\end{align*}
so that $u_k$ is continuous on $\{\alpha \in \{\theta_k, \overline \theta_k\}, \;r\in [0,r_k/2]\} \cap \Omega$,
see \eqref{eqn:vepsboundaryinBeps} and \eqref{eq:u_k_lip_cont}.

Notice also that $u_k$ is continuous 
at $(0,0)\in \R^2$ and $u_k(0,0)=(-1,0)$.
Finally, since $u_k$ in \eqref{eq:u_k_conetto}
takes values in $\mathbb S^1$, the
 determinant of its Jacobian vanishes, 
so that in order to estimate the area contribution
of the graph of $u_k$ in $\Teps \cap \{(r,\alpha):r\in [0,r_k/2],\; \alpha \in [\thetaeps, \thetaepsbar]\}$ it is sufficient to estimate the 
gradient of $u_k$. We have
\begin{align*}
&|\partial_ru_k(r,\alpha)|= |\partial_r \beta_k(r,\alpha)|\leq |\partial_r\alpha_k(r)|+\frac{2\pi}{r_k},\\
&|\partial_\alpha u_k(r,\alpha)|=|\partial_\alpha\beta_k(r,\alpha)|\leq \frac{|\alpha_k(r)|}{\overline\theta_k-\theta_k}+\frac{\pi}{\overline\theta_k-\theta_k}\leq \frac{2\pi}{\overline\theta_k-\theta_k}.
\end{align*}
Therefore 
\begin{align}\label{estimate_area_step5a}
&\int_{0}^{r_k/2}\int_{\theta_k}^{\overline\theta_k}
r|\mathcal M(\nabla u_k)(r,\alpha)|d\alpha dr\leq \int_{0}^{r_k/2}
\int_{\theta_k}^{\overline\theta_k} 
\Big[\frac{r_k}{2}
(1+|\partial_r\beta_k(r,\alpha)|)+|\partial_\alpha \beta_k(r,\alpha)|\Big]
~d\alpha dr\nonumber\\
&\leq o(1)+
\int_{0}^{r_k/2}\int_{\theta_k}^{\overline\theta_k}\left(
\frac{r_k}{2}|\partial_r\alpha_k(r)|+\pi +\frac{2\pi }{\overline\theta_k-\theta_k}\right)d\alpha dr=o(1),
\end{align}
with $o(1)\rightarrow0$ as $k\rightarrow +\infty$. 
Notice that the integral of $|\partial_r\alpha_k(r)|$ with respect to $r$ can be computed via the fundamental integration theorem, since $\alpha_k$ 
is monotone.

In $T_k\cap \{(r,\alpha):  r\in [0,r_k/2],\alpha \in [2\pi-\thetaepsbar,2\pi-\thetaeps]\}$ we set
\begin{align*}
u_k(r,\alpha):=(u_{k1}(r,2\pi-\alpha),-u_{k2}(r,2\pi-\alpha)).
\end{align*}
We now define $u_k$ on $\Teps \cap \{(r,\alpha):  r\in(r_k/2,l),\alpha \in [\thetaeps, \thetaepsbar]\}$. We set
\begin{equation*}
u_k(r,\alpha):=\big(\cos((1-\frac{\overline\theta_k-\alpha}{\overline\theta_k-\theta_k})\overline\theta_k),\sin((1-\frac{\overline\theta_k-\alpha}{\overline\theta_k-\theta_k})\overline\theta_k)\big).
\end{equation*}
Then  $\veps \in \textrm{Lip}(\Teps,\Sone)$, and 
\begin{align*}
& \veps(r,\thetaeps)=(1,0),\qquad \veps(r,\thetaepsbar) = (\cos \thetaepsbar,\sin \thetaepsbar)\quad \text{ for  } r \in (\reps/2,l),
\\
&\partial_ru_k(r,\alpha)= 0,\\
&\partial_\alpha u_k(r,\alpha)=\frac{\overline \theta_k}{\overline \theta_k-\theta_k}\Big(-\sin((1-\frac{\overline\theta_k-\alpha}{\overline\theta_k-\theta_k})\overline\theta_k),\cos((1-\frac{\overline\theta_k-\alpha}{\overline\theta_k-\theta_k})\overline\theta_k)\Big).
\end{align*}
Hence 
\begin{align}\label{estimate_area_step5b}
&\int_{r_k/2}^l\int_{\theta_k}^{\overline\theta_k}
r|\mathcal M(\nabla u_k)(r,\alpha)|d\alpha dr\leq 
\int_{r_k/2}^l\int_{\theta_k}^{\overline\theta_k}\Big(
r+\frac{\overline \theta_k}{\overline \theta_k-\theta_k}\Big)~d\alpha dr=o(1)
\end{align}
as $k\rightarrow +\infty$. 

Finally in  
$T_k\cap \{(r,\alpha): r\in(r_k/2,l), \alpha \in [2\pi-\thetaepsbar,2\pi-\thetaeps]\}$ we set
\begin{align*}
u_k(r,\alpha) := (u_{k1}(r,2\pi-\alpha),-u_{k2}(r,2\pi-\alpha)).
\end{align*}
Similar estimates as in \eqref{estimate_area_step5a}, 
\eqref{estimate_area_step5b} for the area of the graph of $u_k$ hold on 
$T_k\cap \{(r,\alpha): r\in(0,r_k/2), \alpha \in [2\pi-\thetaepsbar,2\pi-\thetaeps]\}$,
$T_k\cap \{(r,\alpha): r\in(r_k/2,l), \alpha \in [2\pi-\thetaepsbar,2\pi-\thetaeps]\}$,
respectively. 

{\it Step 6.} We claim that  
\begin{equation}\label{eqn:integraloutofceps}
\int_{\BallR \setminus 
(C_k\cup T_k)
} \vert \M (\nabla \veps ) \vert dx \longrightarrow \int_\BallR \vert \M (\nabla \vmap)\vert  ~dx \qquad \text{ as }  k \to +\infty,
\end{equation}
where we recall that $C_k\cup T_k = \{(r,\alpha)\in \BallR: r\in[0,\longR), \alpha \in [0,\thetaepsbar]\cup[2\pi-\thetaepsbar,2\pi)\}.$ 

Indeed, on $\BallR \setminus (C_k\cup T_k)$
the maps $\veps$ and $\vmap$ take values in the circle $\Sone$, hence 
\begin{equation*}
\det (\nabla \veps)=0,\qquad \det (\nabla\vmap )
=0,\qquad \text{ in } \BallR \setminus (C_k\cup T_k).
\end{equation*}
Thus  
\begin{equation*}
\int_{\BallR \setminus (C_k\cup T_k)} 
\vert\M(\nabla \veps)-\M(\nabla \vmap)\vert ~ d x 
\leq \sum _{i=1,2}\int_{\BallR \setminus (C_k\cup T_k)} \vert \nabla (\veps_i - \vmap _i) \vert ~d x.
\end{equation*}
From \eqref{eqn:vepsoutofthecone} we have 
\begin{equation}\label{eqn:vepsminusvoutofthecone_rderiv}
\begin{aligned}
&\vert \partial_r (\veps -\vmap) \vert
   = 0 \qquad \textrm{ in }\quad\BallR \setminus (\Balleps \cup \Coneeps \cup \Teps),
\\
&\vert \partial_r (\veps -\vmap) \vert\leq \frac{\pi}{\reps} 
\ \quad \textrm{ in }\quad\Balleps\setminus(\Coneeps \cup \Teps),
\\
&\vert \partial_\alpha (\veps -\vmap) \vert
   = 0 \qquad \text{ in }\quad\BallR \setminus (\Balleps \cup \Coneeps \cup \Teps),
\\
&\vert \partial_\alpha (\veps -\vmap) 
\vert\leq 2 \qquad \textrm{ in }\quad\Balleps \setminus ( \Coneeps\cup \Teps).
\end{aligned}
\end{equation}
Our previous remarks and the fact that $\reps, \thetaeps , (\thetaepsbar -\thetaeps)\to0^+$ as $k \to +\infty$, imply \eqref{eqn:integraloutofceps}. 
\smallskip

{\it Step 7.} 
We know from \eqref{estimate_area_step3}, \eqref{estimate_area_step4}, \eqref{estimate_area_step5a}, and  \eqref{estimate_area_step5b}, that the integral of $\vert\mathcal M(\nabla u_k)\vert$ is infinitesimal as $k\rightarrow 
+\infty$, on the region $(\sourcedisk_{r_k}\cap C_k)\cup T_k$. Therefore it 
remains to compute the area of the graphs of $u_k$ in the region $C_k\setminus {\rm B}_{r_k}$.
We claim that this contribution gives
\begin{equation}\label{eq:this_contribution_will_be}
\lim_{k \to +\infty}\int_{\Coneeps\setminus\Balleps} 
\vert \M (\nabla \veps ) \vert~ dx \leq 2
\FBl(\hstar,\psionestar)=\mathcal A(\psionestar,\Omegahstar).
\end{equation}
To prove this, we start to compute the area of the graph of $\veps$ restricted to $C_k^+\setminus \Balleps$.
From \eqref{vepsonCeps}, \eqref{eqn:tepsderiv}, \eqref{eqn:sepsrderiv} and \eqref{eqn:sepsthetaderiv}, we have
\begin{equation}\label{eq:comput_first_derivatives}
\begin{aligned}
&\partial_r \veps_1=\Big(\frac{\alpha}{\thetaeps}-1\Big)
\teps' 
\hstark'
=
\frac{\longR}{\longR -\reps} 
\Big(\frac{\alpha}{\thetaeps}-1\Big) 
\hstark',
\\
& 
\partial_\alpha \veps_1=\frac{1 + \hstark}{\thetaeps},
\\
&\partial_r \veps_2=
\teps'\Big[\Big(1-\frac{\alpha}{\thetaeps}\Big) 
\hstark' \partial_\scoord \psionekstar 
+ \partial_\axialcoordofcylinder \psionekstar\Big]=
\frac{\longR}{\longR -\reps} 
\Big[ \Big(1-\frac{\alpha}{\thetaeps}\Big) 
\hstark'\partial_\scoord \psionekstar + \partial_\axialcoordofcylinder \psionekstar\Big],
\\
&\partial_\alpha \veps_2
=
-\Big[\frac{1+ \hstark}{\thetaeps}\Big] 
\partial_\scoord\psionekstar,
\\
& 
\partial_r \veps_1
\partial_\alpha \veps_2 - 
\partial_\alpha \veps_1
\partial_r \veps_2 = - \left(
\frac{1+\hstark}{\thetaeps}
\right)\frac{\longR}{\longR -\reps} 
\partial_\axialcoordofcylinder\psionekstar,
\end{aligned}
\end{equation}
where $\hstark'$ denotes the 
derivative of $\hstark$ with respect
to $\axialcoordofcylinder$, 
 $\hstark, \hstark'$ 
are evaluated at 
$\teps(r)$, and the two partial derivatives 
$\partial_\scoord\psionekstar$, 
$\partial_\axialcoordofcylinder \psionekstar$
of $\psionekstar$ with respect to $\scoord, \axialcoordofcylinder$ 
are evaluated at $(\teps(r),
-\seps(r,\alpha))$. Note carefully
that, in the computation of the Jacobian, the terms containing $\partial_\scoord \psionekstar$ cancel each other.

Since $\hstark$ is convex, its derivative is nonincreasing, and therefore
$\int_{\reps}^{\longR} \vert \hstark'\vert~dr < +\infty$.
As a consequence of \eqref{eq:comput_first_derivatives}, from \eqref{eqn:areapolarexpression}, we have 
\begin{equation*}
\begin{aligned}
& \area (\veps,C_k^+\setminus \Balleps)
\\
= &
\int_{\reps}^{\longR}\int_{0}^{\thetaeps}r
 \Bigg\{
1
+\left(\frac{\longR}{\longR-\reps}\right)^2
\left(\frac{\alpha}{\thetaeps}-1\right)^2(\hstark')^2
\\
& +\left(\frac{\longR}{\longR-\reps}\right)^2
\left[
\big(\frac{\alpha}{\thetaeps}-1\big)^2(\hstark')^2(\partial_{\scoord}
\psionekstar)^2 
+
2\big(1-\frac{\alpha}{\thetaeps}\big)\hstark'
\partial_{\scoord} \psionekstar
\partial_{\axialcoordofcylinder} \psionekstar
+ (\partial_{\axialcoordofcylinder} \psionekstar)^2
\right]
\\
&
 +\frac{1}{r^2}\left(
\frac{1+ \hstark}{\thetaeps}
\right)^2
\left(
1+(\partial_{\scoord}\psionekstar)^2
+
\left(\frac{\longR}{\longR-\reps}\right)^2
(\partial_{\axialcoordofcylinder} \psionekstar)^2
\right)\Bigg\}^\frac{1}{2}
~  dr d\alpha,
\end{aligned}
\end{equation*}
where 
$\partial_\scoord \psionekstar$, $\partial_{\axialcoordofcylinder}
\psionekstar$ 
are evaluated at $(\teps(r),-\seps(r,\alpha))$, 
and $\hstark$, $\hstark'$ are evaluated at $\teps(r)$.
Now we use the change of variable \eqref{eqn:changeofvariableT}:
from \eqref{eqn:changeofvariabledeterminant}, we have
\begin{equation*}
\begin{aligned}
& \area (\veps,C_k^+\setminus \Balleps)
\\
=&
\int_{0}^{\longR}
\int^{\hstark(\axialcoordofcylinder)}_{-1}
 \left( \frac{\longR-\reps}{\longR}\right)
 \left( \frac{\thetaeps}{1+\hstark}\right)
 \invr(\axialcoordofcylinder)
 \Bigg\{
1 
+
\big(\frac{\longR}{\longR-\reps}\big)^2
\left(\frac{\invtheta(\axialcoordofcylinder,\scoord)}{\thetaeps}-1\right)^2(\hstark')^2 
\\
&
+\left(\frac{\longR}{\longR-\reps}\right)^2
\Big[
\big(1-\frac{\invtheta(\axialcoordofcylinder,\scoord)}{\thetaeps}\big)^2
(\hstark')^2(\partial_{\scoord} \psionekstar)^2 
+
2\big(1-\frac{\invtheta(\axialcoordofcylinder,\scoord)}{\thetaeps}\big)
\hstark' \partial_{\scoord}\psionekstar
\partial_{\axialcoordofcylinder}\psionekstar
+
(\partial_{\axialcoordofcylinder}\psionekstar)^2
\Big]
\\
&+ \frac{1}{(\invr(\axialcoordofcylinder))^2}
\big(\frac{1+ \hstark}{\thetaeps}\big)^2\Big( 1
+(\partial_{\scoord}\psionekstar)^2
+
\big(\frac{\longR}{\longR-\reps}\big)^2
(\partial_{\axialcoordofcylinder}
\psionekstar
)^2\Big)
\Bigg\}^\frac{1}{2}
~  d\scoord d\axialcoordofcylinder,
\end{aligned}
\end{equation*}
where $\invr(\axialcoordofcylinder)$, 
$\invtheta(\axialcoordofcylinder,\scoord)$ 
are defined in \eqref{eqn:rfn}, \eqref{eqn:thetafn}, $\hstark'$ is evaluated at $w_1$, and $\partial_{w_1} \psi^\star_k$ and $\partial_{w_2} \psi^\star_k$ are evaluated at $(w_1,w_2)$.
Therefore
\begin{equation}\label{last_area}
 \area (\veps,C_k^+\setminus \Balleps)
= 
\int_{0}^{\longR}\int^{\hstark(\axialcoordofcylinder)}_{-1}
\Big\{\textrm{I}_k
+
\textrm{II}_k
+
\textrm{III}_k
+
\textrm{IV}_k
+
\textrm{V}_k
+
\textrm{VI}_k
\Big\}^{\frac{1}{2}} ~d\scoord d\axialcoordofcylinder,
\end{equation}
where
\begin{equation*}
\begin{cases}
\textrm{I}_k = 
\left( \frac{\longR-\reps}{\longR}\right)^2
 \left( \frac{\thetaeps}{1+\hstark}\right)^2(\invr(\axialcoordofcylinder))^2,
\\
\\
\textrm{II}_k = 
 \left( \frac{\thetaeps}{1+\hstark}\right)^2
 \left(1-\frac{\invtheta(\axialcoordofcylinder,\scoord)}{\thetaeps}\right)^2
(\invr(\axialcoordofcylinder))^2(\hstark')^2,
\\
\\
\textrm{III}_k = 
\left( \frac{\thetaeps}{1+\hstark}\right)^2(\invr(\axialcoordofcylinder))^2
\Big[
\big(1-\frac{\invtheta(\axialcoordofcylinder,\scoord)}{\thetaeps}
\big)^2(\hstark')^2(\partial_{\scoord}\psionekstar)^2 
\\
\\
\qquad \qquad 
\qquad \qquad 
\qquad \qquad 
+2\big(1-\frac{\invtheta(\axialcoordofcylinder,\scoord)}{\thetaeps}\big)
\hstark' \partial_{\scoord}\psionekstar
\partial_{\axialcoordofcylinder}
\psionekstar
+
(\partial_{\axialcoordofcylinder}
\psionekstar)^2
\Big],
\\
\\
\textrm{IV}_k  =
\big( \frac{\longR-\reps}{\longR}\big)^2,
\\
\\
\textrm{V}_k  =
\big( \frac{\longR-\reps}{\longR}\big)^2
(\partial_\scoord \psionekstar)^2,
\\ 
\\
\textrm{VI}_k  =
(\partial_{\axialcoordofcylinder}
\psionekstar
)^2.
\end{cases}
\end{equation*}
Since 
$\lim_{k\to \infty} \frac{l-\reps}{l}=1$
and $\lim_{k \to +\infty} \thetaeps=0$, we deduce
from  \eqref{eqn:rfn}, \eqref{eqn:thetafn}, 
\begin{align*}
\lim_{k \to +\infty}\invr(\axialcoordofcylinder) 
= \axialcoordofcylinder, \qquad  
\lim_{k \to +\infty}
\frac{\invtheta(\axialcoordofcylinder,\scoord)}{\thetaeps} = 
\frac{\hstar(\axialcoordofcylinder)-\scoord}{1+\hstar(\axialcoordofcylinder)}.
\end{align*}
Therefore we see that 
$$\int_{0}^{\longR}\int^{\hstark(\axialcoordofcylinder)}_{-1}(\textrm{I}_k)^{\frac12}+(\textrm{II}_k)^{\frac12}dw_2dw_1=o(1),$$
as $k\rightarrow +\infty$.
Moreover 
\begin{equation}\label{eq:III_k}
\int_{0}^{\longR}\int^{\hstark(\axialcoordofcylinder)}_{-1}(\textrm{III}_k)^{\frac12}dw_2dw_1=o(1)
\end{equation}
as $k\rightarrow +\infty$.
Indeed 
we may estimate
$$\int_{0}^{\longR}\int^{\hstark(\axialcoordofcylinder)}_{-1}(\textrm{III}_k)^{\frac12}dw_2dw_1\leq C\theta_k\int_{0}^{\longR}\int^{\hstark(\axialcoordofcylinder)}_{-1}|\hstark'(w_1)||\partial_{w_2}\psi^\star_k(w_1,w_2)|+|\partial_{w_2}\psi^\star_k(w_1,w_2)|dw_2dw_1,$$
and using that $|\hstark'(w_1)|\leq 2k$ (see \eqref{def_h_k}), 
if we assume \eqref{eq:k_theta_k_to_zero}, {\it i.e.}, 
$\theta_k k\rightarrow0$, then 
\eqref{eq:III_k} follows, 
since the $BV$-norm of $\psi_k^\star$ is bounded uniformly with respect to  $k$.

Hence, from \eqref{last_area},
\begin{align}\label{eq:almost_the_end}
\area (\veps,C_k^+\setminus \Balleps)
&\leq
\int_{0}^{\longR}\int^{\hstark(\axialcoordofcylinder)}_{-1}
\Big\{
\textrm{IV}_k
+
\textrm{V}_k
+
\textrm{VI}_k
\Big\}^{\frac{1}{2}} ~d\scoord d\axialcoordofcylinder+o(1)\nonumber\\
&\leq\int_0^\longR \int^{\hstark(\axialcoordofcylinder)}_{-1}
\sqrt{1+(\partial_{\axialcoordofcylinder}
	\psionekstar)^2 +(\partial_{w_2}
	\psionekstar)^2}~d\scoord dw_1+o(1)\nonumber\\
&=\mathcal A(\psi^\star_k, SG_{h^\star}\cap R_l)+o(1)=\frac12\mathcal A(\psi^\star_k, SG_{h^\star})+o(1)
\end{align} 
as $k\rightarrow +\infty$.
Then taking the limit as $k \to +\infty$ in \eqref{eq:almost_the_end}, and using Lemma \ref{lem:properties_of_psi_m} (iii),  we get 
\begin{equation}\label{eq:step_8}
\lim_{k \to +\infty} \area(\veps, C_k^+\setminus \Balleps)
 \leq\mathcal A(\psionestar, \Omegahstar)=\FB(\hstar,\psionestar),
\end{equation}
where the last equality follows from \eqref{eqn:FAequalareaofhstar}.

\smallskip
{\it Step 8.} Conclusion.
Notice that $\veps \in {\rm Lip}(\Omega, \R^2)$,
and $u_k \to \vmap$  in $L^1(\BallR,\R^2)$. 
Inequality 
\eqref{eq:upper_bound_recovery} follows from 
\eqref{eqn:integraloutofceps} (which gives the term $\int_\Omega|\mathcal M(\nabla u)|dx$), from  \eqref{eq:this_contribution_will_be} (which gives the second term in \eqref{eq:upper_bound_recovery}), and from estimates \eqref{estimate_area_step3}, \eqref{estimate_area_step4}, \eqref{estimate_area_step5a}, and \eqref{estimate_area_step5b}, showing that all the other contributions are negligible. 

%%%%%%%%%%%%%%%%%%%%%%%%%%%%%%%%%%%%%%%%%%%%%%%%%%%%%%%%%%%%%%%%%%%%%%%
\section*{Acknowledgements}
%%%%%%%%%%%%%%%%%%%%%%%%%%%%%%%%%%%%%%%%%%%%%%%%%%%%%%%%%%%%%%%%%%%%%%%
The first and third authors acknowledge the support
of the INDAM/GNAMPA.
The first two authors are
grateful to ICTP (Trieste), where part of this paper was written.
The first and third authors also acknowledge the partial financial support of the F-cur project number 2262-2022-SR-CONRICMIUR$_{-}$PC-FCUR2022$_{-}002$ of the University of Siena, and the of the PRIN project 2022PJ9EFL "Geometric Measure Theory: Structure of Singular
Measures, Regularity Theory and Applications in the Calculus of Variations'', PNRR Italia Domani, funded
by the European Union via the program NextGenerationEU, CUP B53D23009400006.

\end{document}